\documentclass[12pt]{article}
\usepackage{color}
\usepackage{amsmath}
\usepackage{amsfonts}
\usepackage{amsthm}
\usepackage{amssymb}
\usepackage{MnSymbol}
\usepackage[applemac]{inputenc}
\usepackage{enumerate}
\makeatletter

\newtheoremstyle{mystyle}{}{}{\slshape}{2pt}{\scshape}{.}{ }{} 

\newtheorem{thm}{Theorem}[section]
\newtheorem{cor}[thm]{Corollary}

\newtheorem{prop}[thm]{Proposition}
\newtheorem{lemme}[thm]{Lemma}
\newtheorem{fait}[thm]{Fact}

\newtheorem{obs}[thm]{Observation}
\theoremstyle{definition}
\newtheorem{defi}[thm]{Definition}
\theoremstyle{mystyle}
\newtheorem{ex}[thm]{Example}

\theoremstyle{remark}
\newtheorem{rem}[thm]{Remark}

\newcommand{\un}{\underline}
\newcommand{\llg}{\langle}
\newcommand{\rrg}{\rangle}
\newcommand{\impl}{\rightarrow}

\newcommand{\monster}{\mathfrak C}

\DeclareMathOperator{\tp}{tp}

\DeclareMathOperator{\val}{val}

\title{Distal and non-distal NIP theories}
\author{Pierre Simon \\ \'Ecole Normale Sup\'erieure, Paris}

\begin{document}
\maketitle
\begin{abstract}
We study one way in which stable phenomena can exist in an $NIP$ theory. We start by defining a notion of `pure instability' that we call `distality' in which no such phenomenon occurs. O-minimal theories and the p-adics for example are distal. Next, we try to understand what happens when distality fails. Given a type $p$ over a sufficiently saturated model, we extract, in some sense, the stable part of $p$ and define a notion of stable-independence which is implied by non-forking and has bounded weight.
\end{abstract}
%
%

\tableofcontents
\section{Introduction}

We study one way in which stability and order can interact in an $NIP$ theory. More precisely, we are interested in the situation where stability and order are intertwined. We start by giving some very simple examples illustrating what we mean.

Consider $M_0\models$ DLO. A type of $S_1(M_0)$ is determined by a cut in $M_0$ and two types corresponding to different cuts are orthogonal. If we take now $M_1$ a model of some o-minimal theory, still a 1-type is determined by a cut, but in general, types that correspond to different cuts are not orthogonal. However this is true over indiscernible sequences in the following sense: assume $\llg a_t:t<\omega+\omega\rrg \subset M_1$ is an indiscernible sequence. By $NIP$, the sequences of types $\llg \tp(a_t/M_1) : t<\omega \rrg$ and $\llg \tp(a_{\omega+t}/M_1): t<\omega\rrg$ converge in $S(M_1)$. Then the two limit types are orthogonal (this follows from dp-minimality, see \ref{omin}). An indiscernible sequence with that property will be called \emph{distal}\footnote{Thanks to Itay Kaplan for suggesting the name.}. A theory is distal if all indiscernible sequences are distal. So any o-minimal theory is distal.

Distality for an indiscernible sequence can be considered as an opposite notion to that of total indiscernibility.

Let now $M_2$ be a model of ACVF (or any other C-minimal structure) and consider an indiscernible sequence $(a_i)_{i<\omega}$ of elements from the valued field sort. Two different behaviors are possible: either the sequence is totally indiscernible, this happens if and only if $\val(a_i-a_j)=\val(a_{i'}-a_{j'})$ for all $i\neq j$, $i'\neq j'$, or the sequence is distal. Again, this will follow from the results in Section 2, but could be proved directly. So $M_2$ is neither stable nor distal; the two phenomena exist but do not interact in a single indiscernible sequence of points.


Consider now a fourth structure (a `colored order') $M_3$ in the language $L_3=\{\leq, E\}$: $M_3$ is totally ordered by $\leq$ and $E$ defines an equivalence relation, each $E$ class being dense co-dense with respect to $\leq$. Now an indiscernible sequence of elements from different $E$ classes is neither totally indiscernible nor distal. Given two limit types $p_x$ and $q_y$ of different cuts in such a sequence, the type $p_x \cup q_y$ is consistent with $xEy$ and with $\neg xEy$. Here it is clear that the `stable part' of a type should be its $E$-class.

The idea behind the work in this paper is that every ordered indiscernible sequence in an $NIP$ theory should look like a colored order: there is an order for which different cuts are orthogonal and something stable on top of it which does not see the order (see Section 3).

\subsubsection{A word about measures}

Keisler measures will be used a little in this work, however the reader not familiar with them can skip all parts referring to measures without harm. For this reason, we will be very brief in recalling some facts about them and refer the reader to \cite{NIP2} and \cite{NIP3}. They however give some understanding of the intuition behind some definitions and results. We explain this now.

A Keisler measure (or simply a measure) is a Borel probability measure on a type space $S_x(A)$. Basic definitions for types (non-forking, invariance, coheir, Morley sequence etc.) generalize naturally to measures (see \cite{NIP2} and \cite{NIP3}). Of interest to us is the notion of \emph{generically stable measure}. A measure is generically stable if it is both definable and finitely satisfiable over some small set. Equivalently, its Morley sequence is totally indiscernible. Such measures are defined and studied by Hrushovski, Pillay and the author in \cite{NIP3}. Furthermore, it is shown in \cite{Finding} that some general constructions give rise to them, and in this sense they are better behaved than the more natural notion of generically stable type.

This paper can be considered as an attempt to understand where generically stable measures come from. What stable phenomena do generically stable measures detect? What does the existence of generically stable measures in some particular theory tell us about types? The first test question was: Can we characterize theories which have non-trivial generically stable measures? Here ``non-trivial" means ``non-smooth": a measure is smooth if it has a unique extension to any bigger set of parameters. This question is answered in Section 2: a theory has a non-smooth generically stable measure if and only if it is not distal.

The main tool at our disposal to link measures to indiscernible sequences is the construction of an average measure of an indiscernible segment (see \cite{NIP3} Lemma 3.4 or \cite{Finding} Section 3 for a more elaborate construction). Such a measure is always generically stable. The intuition we suggest is that the `order' component of the sequence is evened out in the average measure and only the `stable' component remains.

\subsubsection{Organization of the paper and main results}

The paper is organized as follows. The first section contains some basic facts about $NIP$ theories and Keisler measures. We give a number of definitions concerning indiscernible sequences and some basic results illustrating how we can manipulate them. Section 2 studies distal theories. They are defined as theories in which every indiscernible sequence is distal, as explained above. We show that this condition can also be seen through invariant types and generically stable measures. The main results can be summarized by the following theorem.

\begin{thm}
Let $T$ be NIP. Then the following are equivalent:
\begin{itemize}
\item $T$ is distal,
\item Any two invariant types that commute are orthogonal,
\item All generically stable measures are smooth.
\end{itemize}
Furthermore, it is enough to check any one of those conditions in dimension 1.
\end{thm}

As a consequence, o-minimal theories and the p-adics are distal as are more generally any dp-minimal theory with no generically stable type.

Section 3 can be read almost independently of the previous one: it contains a study of the intermediate case of an $NIP$ theory that is neither stable nor distal. We deal with the problem of understanding to what extend non-distality is witnessed by stable-like interactions between tuples. If $M$ is a $|T|^+$-saturated model, we define a notion of s-independence denoted $a\downfree^s_M b$ which is sym\-metric, is implied by forking-independence and has bounded weight. We use it to show that two commuting types behave with respect to each other like types in a stable theory (we recover some definability and uniqueness of the non-forking extension). The guiding intuition is that of the colored order where elements have a well defined \emph{stable part} (the image in the quotient) and in that case $a\downfree^s_M b$ means that the stable parts are independent. We do not attempt to give any meaning to the `stable part' of a type in general, and we do not even expect there to be a possible meaning for it. We find that the intuition ``s-independence corresponds to independence of stable parts" is useful in understanding those results. Of course, it may turn out some day to be misleading.

As an application of those ideas, we prove the following `finite-co-finite theorem' (Theorem \ref{whatcanhappen}) and give an application of it to the study of externally definable sets.
\begin{thm}[Finite-co-finite theorem]
Assume that $T$ is NIP. Let $I=I_1+I_2+I_3$ be indiscernible, $I_1$ and $I_3$ being infinite. Assume that $I_1+I_3$ is $A$-indiscernible and take $\phi(x;a)\in L(A)$, then the set $B=\{ b \in I_2 : \models \phi(b;a)\}$ is finite or co-finite.
\end{thm}

The last section defines a class of theories \textemdash~ called \emph{sharp} \textemdash~ in which (intuitively) the stable part of types is witnessed by generically stable types. More precisely, over a $|T|^+$-saturated model $M$, every tuple is s-dominated by the realization of a generically stable type. We give a criterion for sharpness which only involves looking at indiscernible sequences of elements (not tuples). In particular, any dp-minimal theory is sharp.
\\

Our Bible concerning $NIP$ theories are Shelah's papers \cite{Sh715}, \cite{Sh783}, \cite{Sh863}, \cite{Sh900} and \cite{Sh950}. We will however use ideas only from the first two. All the basic insights about indiscernible sequences were taken from there (although the important result on shrinking indiscernible sequences originates in \cite{BB}).

In fact, we realized after having done most of this work that the idea of `domination' for indiscernible sequences was already in Shelah's work: in Section 2 of \cite{Sh783} in a slightly different wording and with a very different purpose. The main additional ingredient in Section 3 is the external characterization of domination (\ref{extstablebase}) which allows us to say something about points outside of the indiscernible sequence and then to generalize to the invariant type setting.

An important property of stable theories sometimes referred to as the \emph{Shelah reflection principle} says roughly that non-trivial relationships between a realization of a type $p$ and some other point are reflected inside realizations of $p$. Internal concepts (only considering realizations of $p$) often imply external properties (involving the whole structure). For example regularity implies weight one. There is some evidence now that this principle is already true in $NIP$ theories. See \cite{DepPairs} for an example (weak stable embeddedness).

In this paper we will use this principle for indiscernible sequences: a property involving only the indiscernible sequence itself or extensions of it usually implies properties of the indiscernible sequence with respect to points outside (the same way total indiscernibitily implies that the trace of every definable set is finite or co-finite). See Lemma \ref{defitrans} and Proposition \ref{extstablebase}.

\subsubsection*{Acknowledgment}

I would like to thank Elisabeth Bouscaren for having patiently listened to the expositions of many erroneous proofs of results here and for helping finding mistakes. Thanks to Udi Hrushovski for interesting remarks on earlier versions of this work and for his encouragements to pursue it. Thanks also to Tom Scanlon and the anonymous referee for a number of helpful comments.

\subsection{Preliminaries}

We work with a complete theory $T$, in a language $L$. We let $\monster$ denote a monster model of $T$.

We will often denote sequences of tuples by $I,J,...$. Index sets of families or sequence might be named $\mathcal I,\mathcal J,...$. 

If $I$ is an indiscernible sequence and $A$ a set of parameters, let $\lim(I/A)$ let $I$ be the limit type of $I$ over $A$ defined as follows: if $I=(a_t)_{t\in \mathcal I}$, and $\phi(x;d)\in L(A)$, then $\phi(x;d)\in \lim(I/A)$ if and only if for some $t_0\in \mathcal I$, $\models \phi(a_t;d)$ holds for all $t\geq t_0$. Recall that a theory is NIP if and only if $\lim(I/A)$ is a complete type for every $I$ and $A$. By $\lim(I)$, we mean the global type $\lim(I/\monster)$.

\medskip
\noindent
\textbf{Assumption :} Throughout the paper, we assume that the theory $T$ is NIP.

\medskip
Let $M$ be a $\kappa$-saturated model (for some $\kappa > |T|$). If $A\subseteq M$, $|A|<\kappa$, then a type $p\in S(M)$ is $A$-invariant if for $a\models p$ and any tuples $b,b'\in M$, $b\equiv_A b' \impl ba \equiv b'a$. We will sometimes say simply that $p$ is an invariant type, without specifying $A$. Note that an invariant type has a natural extension to any larger set $B\supset M$ that we will denote by $p|_B$. We use the same notation to denote the restriction of $p$ to $B$, when $B\subset M$.

Let $\mathcal I$ be a linear order. A Morley sequence indexed by $\mathcal I$ of an invariant type $p$ over some $B\supseteq A$ is a sequence $(a_t)_{t \in \mathcal I}$ such that $a_t \models p|_{B\cup a_{<t}}$ for every $t$. All Morley sequences of $p$ over $B$ indexed by $\mathcal I$ are $B$-indiscernible and have the same type over $B$; when $B=M$, we will denote that type by $p^{(\mathcal I)}$.

If $p_x$ and $q_y$ are two types over the $\kappa$-saturated model $M$ and $p$ is invariant, we can define the product $p_x\otimes q_y$ as the element of $S_{xy}(M)$ defined as $\tp(a,b/M)$ where $b\models q_y$ and $a\models p_x|_{Mb}$. If $q$ is also an invariant type, then $p_x\otimes q_y$ is invariant. In this case, we can also build the product $q_y\otimes p_x$. When the two products are equal, we say that $p$ and $q$ commute.

Note that $\otimes$ is associative. In particular if $p$ and $q$ commute with $r$, then $r$ commutes with $p\otimes q$.

\begin{defi}
Two types $p_x,q_y$ over the same domain $A$ are \emph{weakly orthogonal} if $p_x\cup q_y$ defines a complete type in two variables over $A$.

If $p_x,q_y\in S(M)$ are invariant over $A\subset M$ ($M$ is $\kappa$-saturated and $|A|<\kappa$), then we say that $p_x$ and $q_y$ are \emph{orthogonal} if they are weakly orthogonal. This implies that $p|_B$ and $q|_B$ are also weakly orthogonal for any $B\supseteq M$.
\end{defi}

Recall the notion of generically stable type from \cite{Sh715} and \cite{NIP2}: an invariant type $p\in S(M)$ is generically stable if it is both definable and finitely satisfiable in some small model $N\subset M$. Equivalently, its Morley sequence is totally indiscernible.

\subsubsection{Measures}

As we mentioned in the introduction, we will not recall all definitions concerning measures. Instead, we refer the reader to \cite{NIP2} and \cite{NIP3}. The latter paper contains in particular the definition of a generically stable measure. Also the introduction of \cite{Finding} contains a concise account of the definitions and basic results we will need, but without proofs.

We will need to extend the definition of weakly orthogonal for a type and a measure: if $\mu_x$ is a measure over $A$ and $p_y$ a type over the same $A$, we say that they are \emph{weakly orthogonal} if $\mu_x$ has a unique extension to a measure over $Ab$, where $b\models p_y$.

We also recall the following from \cite{Finding}: if $M$ is a model, a measure $\mu \in \mathcal M_x(M)$ is \emph{smooth} if it has a unique extension to any $N\supset M$. For any formula $\phi(x,d)$, $d\in \monster$, let $\partial_M \phi$ denote the closed subset of $S_x(M)$ consisting of types $p$ such that there are $a,a'$ two realizations of $p$ satisfying $\phi(a,d)\wedge \neg \phi(a',d)$.

\begin{fait}[Lemma 4.1 of \cite{Finding}]
The measure $\mu\in \mathcal M_x(M)$ is smooth if and only if $\mu(\partial_M \phi)=0$ for all formulas $\phi(x,d)$, $d\in \monster$.
\end{fait}

\subsubsection{Indiscernible sequences and cuts}

The notation $I=I_1+I_2$ means that the sequence $I$ is the concatenation of the sequences $I_1$ and $I_2$: $I_1$ is an initial segment of $I$ and $I_2$ the complementary final segment. This operation is associative, and we will also use it to denote the concatenation of three or more sequences. It may be the case that one of the sequences is finite. In particular, when $b$ is a tuple, we may write $I_1+b+I_2$ to denote $I_1+\langle b\rangle+I_2$ where $\langle b\rangle$ is the sequence of length 1 whose only member is $b$.

If $I=I_1+I_2$, we will say that $(I_1,I_2)$ is a \emph{cut} of $I$.

By the EM-type (over $A$) of an indiscernible sequence $I=\llg a_i : i\in \mathcal I \rrg$, we mean the family $(p_n)_{n<\omega}$, where $p_n\in S_n(A)$ is the type of $(a_{\sigma(k)})_{k<n}$ for $\sigma : n \rightarrow \mathcal I$ any increasing embedding.

We now introduce a number of definitions that will be useful for handling indiscernible sequences.

\begin{defi}[Cuts]
If $J\subset I$ is a convex subsequence, a cut $\mathfrak c=(I_1,I_2)$ is said to be \emph{interior} to $J$ if $I_1 \cap J$ and $I_2 \cap J$ are infinite.

A cut is \emph{Dedekind} if both $I_1$ and $I_2^*$ ($I_2$ with the order reversed) have infinite cofinality.

If $\mathfrak c=(I_1,I_2)$ and $\mathfrak d=(J_1,J_2)$ are two cuts of the same sequence $I$, then we write $\mathfrak c \leq \mathfrak d$ if $I_1 \subseteq J_1$.

We write $(I_1',I_2') \unlhd (I_1,I_2)$ if $I_1'$ is an end segment of $I_1$ and $I_2'$ an initial segment of $I_2$.
A \emph{polarized cut} is a pair $(\mathfrak c,\varepsilon)$ where $\mathfrak c$ is a cut $(I_1,I_2)$ and $\varepsilon \in \{1,2\}$ is such that $I_\varepsilon$ is infinite. We will write the polarized cut $\mathfrak c^{-}$ if $\varepsilon = 1$ and $\mathfrak c^+$ if $\varepsilon=2$.

Given a polarized cut $\mathfrak c^{\bullet}=((I_1,I_2),\varepsilon)$ and a set $A$ of parameters, we can define the \emph{limit type} of $\mathfrak c^{\bullet}$ denoted by $\lim(\mathfrak c^\bullet/A)$ as the limit type of the sequence $I_1$ or $I_2^*$ depending on the value of $\varepsilon$.

If a cut $\mathfrak c$ has a unique polarization, or if we know both polarizations give the same limit type over $A$, we will write simply $\lim(\mathfrak c/A)$.

If $\mathfrak c=(I_1,I_2)$ is a cut, we say that the tuple $b$ \emph{fills} the cut $\mathfrak c$ if $I_1+b+I_2$ is indiscernible. Similarly, if $\bar b$ is a sequence of tuples, we will say that $\bar b$ fills $c$ if the concatenation $I_1+\bar b+I_2$ is indiscernible.
\end{defi}

The following definition is from \cite{Sh715}.
\begin{defi}
Let $\mathfrak c=(I_1,I_2)$ be a Dedekind cut. A set $A$ \emph{weakly respects} $\mathfrak c$ if $\lim(\mathfrak c^+/A)=\lim(\mathfrak c^-/A)$. It \emph{respects} $\mathfrak c$ if for every finite $A_0\subseteq A$, there is $I'_1$ cofinal in $I_1$ and $I'_2$ coinitial in $I_2$ such that $I'_1+I'_2$ is indiscernible over $A_0$.
\end{defi}

Note that $\lim(\mathfrak c^\bullet)=\lim(\mathfrak c^\bullet/\monster)$ is an invariant type, in fact finitely satisfiable over the sequence $I$.

If $\mathfrak c_1$ and $\mathfrak c_2$ are two distinct polarized cuts in an indiscernible sequence $I$ then $\lim(\mathfrak c_1)$ and $\lim(\mathfrak c_2)$ commute: $\lim(\mathfrak c_1)_x\otimes \lim(\mathfrak c_2)_y=\lim(\mathfrak c_2)_y\otimes \lim(\mathfrak c_1)_x$. More precisely $\phi(x,y) \in \lim(\mathfrak c_1)_x\otimes \lim(\mathfrak c_2)_y$ if and only if for some $J_1$ cofinal in $\mathfrak c_1$ and $J_2$ cofinal in $\mathfrak c_2$, $\phi(a,b)$ holds for $(a,b) \in J_1\times J_2$. 

\begin{defi}[Polycut]
A polycut is a sequence $(\mathfrak c_i)_{ i\in \mathcal I}$ of pairwise distinct cuts.

The definitions given for cuts extend naturally to polycuts: a polarized polycut is a family of polarized cuts. If $\mathfrak c=(\mathfrak c_i)_{i\in \mathcal I}$ is a polarized polycut, then we define $\lim(\mathfrak c)=\bigotimes_{i\in \mathcal I} \lim(\mathfrak c_i)$. It is a type in variables $(x_i)_{i\in \mathcal I}$. A tuple $(a_i)_{i\in \mathcal I}$ \emph{fills} $\mathfrak c$ if the sequence $I$ with all the points $a_i$ added in their respective cut is indiscernible. Note that this is stronger than asking that each $a_i$ fills $\mathfrak c_i$.
\end{defi}

\begin{defi}[$I$-independent]\label{defi_Iind}
Let $I$ be a dense indiscernible sequence, $\mathfrak c_1,..,\mathfrak c_n$ pairwise distinct cuts in $I$ and $a_1,..,a_n$ filling those cuts, then $a_1,..,a_n$ are independent over $I$ (or $I$-independent) if the tuple $(a_1,...,a_n)$ fills the polycut $(\mathfrak c_1,...,\mathfrak c_n)$.
\end{defi}

We will use the notation $a \downfree_I b$ to mean that $a$ and $b$ are independent over $I$, {\it i.e.}, that $I\cup\{a\}\cup\{b\}$ remains indiscernible (where $I\cup \{a\} \cup \{b\}$ is ordered so that $a$ and $b$ fall in their respective cuts). Note that this is a symmetric notion.

The proofs in this paper will involve a lot of constructions with indiscernible sequences. We list here the basic results and ideas we will need for that. We tried to encapsulate in lemmas some constructions that we will use often. However, in some cases, the lemmas will not fit exactly our needs. The reader should therefore bear in mind the principles of those constructions more than the statements themselves. The constructions are grouped in three parts: shrinking, expanding and sliding.

\subsection{Shrinking}\label{sec_shrinking}

We start with the very important results concerning shrinking of indiscernibles. We give the statement as in \cite[Section 3]{Sh715}. See also \cite{Adler}.

\begin{defi}
A finite convex equivalence relation on $\mathcal I$ is an equivalence relation $\sim$ on $\mathcal I$ which has finitely many classes, all of which are convex subsets of $\mathcal I$.
\end{defi}

\begin{prop}[Shrinking indiscernibles]\label{shrinking1}
Let $A$ be any set of parameters and $( a_t)_{t\in \mathcal I}$ be an $A$-indiscernible sequence. Let $d$ be any tuple. Let $\phi(x_d;y_0,..,y_{n-1},z)$ be a formula. There is a finite convex equivalence relation $\sim$ on $\mathcal I$ such that given: 

-- $t_0<\ldots<t_{n-1}$ in $\mathcal I$;

-- $s_0<\ldots<s_{n-1}$ in $\mathcal I$ with $t_k \sim s_k$ for all $k$;

-- $b\in A^{|t|}$,

\noindent
we have $\phi(d;a_{t_0},..,a_{t_{n-1}},b) \leftrightarrow \phi(d;a_{s_0},...,a_{s_{n-1}},b)$.

Furthermore, there is a coarsest such equivalence relation.
\end{prop}

Often we will apply this with $A=\emptyset$, in which case $b$ does not appear.

We elaborate a little bit on this statement. We fix some parameter set $A$, sequence $I$, tuple $d$ and formula $\phi(x_d;y_0,...,y_{n-1},z)$ such that $I$ is indiscernible over $A$. Consider the coarsest equivalence relation $\sim$ satisfying the conclusion of Proposition \ref{shrinking1}.

The relation $\sim$ induces a partition of the sequence $\mathcal I$ into convex equivalence classes: $\mathcal I=\mathcal I_1+\ldots+\mathcal I_T$. We define also the corresponding partition of $I$ as $I=I_1+\ldots + I_T$.

The $T-1$ cuts $(I_1+\ldots+I_{k-1},I_{k}+\ldots+I_T)$, for $k<T$, will be called the cuts \emph{induced by }$(d,\phi)$\emph{ on }$I$ (over $A$). For the purpose of this section, we will denote them by $\textsf {cut}_I(d,\phi;0)< \ldots < \textsf {cut}_{I}(d,\phi;T-1)$. Here $A$ is implicit to simplify the notation. Let also $\textsf T_I(d,\phi)=T$ be the number of such cuts.

Let $\mathcal F(n,T)$ be the set of non-decreasing functions from $n$ to $T$. For any $f\in \mathcal F(n,T)$ and $b\in A^{|t|}$, there is a truth value $\varepsilon_{d,\phi;I}(f,b)$ such that $\phi(d;a_{t_0},\ldots,a_{t_{n-1}},b)$ has truth value $\varepsilon_{d,\phi;I}(f,b)$ for any $t_0 < \ldots < t_{n-1}$ with $t_k \in I_{f(k)}$ for all $k<T$.

To summarize, the tuple $d$ the sequence $I$ and the set $A$ being fixed, we have associated, to any formula $\phi(x_d;y_0,\ldots,y_{n_{\phi}-1},z)$ an integer $\textsf T_I(d,\phi)$, cuts $\textsf {cut}_i(d,\phi;I)$ for $i<\textsf T_I(d,\phi)$ and a function $\epsilon_{I}(d,\phi): \mathcal F(n,\textsf T_I(d,\phi))\times A^{|t|} \rightarrow \{\top, \bot\}$. This data completely describes the type of $d$ over $IA$.

\begin{lemme}\label{cofinal}
Let $I=( a_t)_{t\in \mathcal I}$ be $A$-indiscernible with $\mathcal I$ of cofinality at least $|T|^+$, then for any finite tuple $d$, there is an end segment $I'$ of $I$ that is indiscernible over $Ad$. 
\end{lemme}
\begin{proof}
Simply take $I'$ to be to the right of all the cuts $\textsf {cut}_I(d,\phi;i)$.
\end{proof}

Usually, when we consider the type of a tuple $d$ over an indiscernible sequence $I$, we are not concerned with the exact type, but only with the number of cuts induced by $d$ on $I$ and their relative position with respect to each other. We now define a notion of similarity between types which makes this precise.

Let $d$ be a tuple and $I=(a_t)_{t\in \mathcal I}$ an indiscernible sequence. We define a structure $I_{[d]}$ as follows: its universe is $\{a_t :t\in \mathcal I\}$, the language contains a binary relation $<_I$ interpreted as the order on $I$ and for each formula $\phi(x_d;y_0,\ldots,y_n)\in L$, a $n$-ary predicate $R_\phi(y_0,\ldots,y_{n-1})$ which holds on $(a_{t_0},\ldots,a_{t_{n-1}})$ if and only if $\models \phi(d;a_{t_0},\ldots,a_{t_{n-1}})$.

\begin{defi}
Let $I,J$ be two indiscernible sequence and $d,d'$ two tuples of the same length. We say that $\tp(d/I)$ and $\tp(d'/J)$ are \emph{similar} if $I_{[d]}\equiv J_{[d']}$.

If $I$ and $J$ are indiscernible over $A$, we say that the two types are similar \emph{over $A$} if they are similar, in the expanded language $L(A)$.
\end{defi}

Note that in particular, if $\tp(d/I)$ and $\tp(d'/I')$ are similar over $A$, then $\tp(d/A)=\tp(d'/A)$ and the EM-types of $I$ and $J$ over $A$ are the same.

The structure $I_{[d]}$ is bi-interpretable with the structure having same universe, whose language contains the binary relation $<_I$ and for each cut $\textsf {cut}_i(d,\phi;I)$ a unary predicate interpreted as the left-piece of the cut. When $I$ and $J$ are densely ordered without endpoints (which will almost always be the case), then $\tp(d/I)$ and $\tp(d'/J)$ are similar over $A$ if and only if  for all formula $\phi$ and $\psi$ as above, the following conditions are satisfied:

-- $\textsf T_I(d,\phi)=\textsf T_{J}(d',\phi)$;

-- $\epsilon_I(d,\phi) = \epsilon_{J}(d',\phi)$;

-- for all $i<\textsf T_{I}(d,\phi)$, the cuts $\textsf {cut}_I(d,\phi;i)$ and $\textsf {cut}_{J}(d',\phi;i)$ are either both of infinite cofinality from the left (resp. right) or both of finite cofinality from the left (resp. right);

-- for all $i<\textsf T_{I}(d,\phi)$ and $j<\textsf T_{I}(d,\psi)$, we have $\textsf {cut}_I(d,\phi;i) < \textsf {cut}_{I}(d,\psi;j)$ if and only if $\textsf {cut}_{I'}(d',\phi;i)<\textsf {cut}_{I'}(d',\psi;j)$;

-- there are infinitely many elements in $I$ between the cuts $\textsf {cut}_I(d,\phi;i)$ and $\textsf {cut}_{I}(d',\psi;j)$ if and only if there are infinitely many elements in $J$ between the cuts $\textsf {cut}_{J}(d',\phi;i)$ and $\textsf {cut}_{J}(d',\psi;j)$.

\begin{lemme}\label{shrinking}
Let $I$ be a dense indiscernible sequence over a set $A$, and $d$ a tuple, then there is $I' \subset I$ of size at most $|T|+|d|$ such that $\tp(d/I')$ and $\tp(d/I)$ are similar over $A$.
\end{lemme}
\begin{proof}
This is immediate by L\"owenheim-Skolem.
%
%
%
%
\end{proof}

\subsection{Expanding}

Let $I$ be an indiscernible sequence over some set $A$, and $d$ any tuple. We now study how one can extend $I$ to some bigger sequence $I'$ maintaining the similarity type of $\tp(d/I)$ over $A$.

First, if $I$ is endless, there is a limit type $\lim(I)$ as defined above. If $J$ realizes a Morley sequence of that type over $Ad$, then $I+J^*$ is indiscernible, where $J^*$ is the sequence $J$ with the opposite order. Also $\tp(d/I+J^*)$ is similar to $\tp(d/I)$ over $A$.

Consider now a cut $\mathfrak c=(I_1,I_2)$ of $I$. If $I_1$ is endless, then we can similarly consider $K$ a Morley sequence of $\lim(I_1)$ over $IA$. Then $I_1+K^*+I_2$ is indiscernible and $\tp(d/I_1+K^*+I_2)$ is similar to $\tp(d/I_1+I_2)$. If $I_2$ has no first element, then we can similarly extend by realizing a Morley sequence in $\lim(I_2^*)$. Note that unless the cut $\mathfrak c$ is induced by $(d,\phi)$ on $I$ for some formula $\phi$, then $\lim(I_1/IAd)=\lim(I_2^*/IAd)$.

If we want to extend the sequence $I$ by adding elements in different cuts, we can iterate the above procedure. Note that the order in which we chose the cuts does not matter since the different limit types commute with each other.

We therefore conclude the following lemma.

\begin{lemme}\label{expanding}
Let $I=(a_i)_{i\in \mathcal I}$ be an indiscernible sequence over some set $A$. Assume $\mathcal I$ is dense without endpoints. Let $d$ be any tuple and let $\mathcal J\supset \mathcal I$ be any linearly ordered set extending $\mathcal I$. Then there are tuples $(a_i)_{i\in \mathcal J \setminus \mathcal I}$ such that the sequence $J=(a_i)_{i\in \mathcal J}$ is indiscernible over $A$ and $\tp(d/J)$ is similar to $\tp(d/I)$ over $A$.
\end{lemme}

\subsection{Sliding}

We are now concerned with the situation where we have $A$, $I$ and $d$ as above, and we want to produce some $d'$ with the same similarity type as $d$, but such that the cuts induced by $d'$ are different from those induced by $d$. We see this as \emph{sliding} the point $d$ along the sequence.

We  state the result in a slightly more general form..

\begin{lemme}\label{sliding1}
Let $I$, $J$ be two dense sequences, indiscernible over some set $A$. Assume they have no endpoints and have the same EM-type over $A$. Let $d$ be any tuple. For any formula $\phi$ such that $\textsf {cut}_I(d,\phi;i)$ is well defined, pick a cut $\mathfrak d(\phi;i)$ of $J$ such for any $\phi$, $\psi$, $i$, $j$ for which this makes sense:

-- the cuts $\textsf {cut}_I(d,\phi;i)$ and $\mathfrak d(\phi;i)$ are either both of infinite cofinality from the left (resp. right) or both of finite cofinality from the left (resp. right);

-- we have $\mathfrak d(\phi;i)<\mathfrak d(\psi;j)$ if and only if $\textsf {cut}_I(d,\phi;i) < \textsf {cut}_{I}(d,\psi;j)$;

--there are infinitely many elements in $J$ between the cuts $\mathfrak d(\phi;i)$ and $\mathfrak d(\psi;j)$ if and only if there are infinitely many elements in $I$ between the cuts $\textsf {cut}_I(d,\phi;i)$ and $\textsf {cut}_{I}(d',\psi;j)$.

Then there is a point $e$ such that $\tp(e/J)$ is similar to $\tp(d/I)$ over $A$ and $\textsf {cut}_{J}(e,\phi;i)=\mathfrak d(\phi;i)$ for any $\phi$ and $i$.
\end{lemme}
\begin{proof}
This translates into finding $e$ with a prescribed type $p(x)$ over $AJ$. Let $\theta(x;\bar m)\in p(x)$, $\bar m\subset J$. Also we may assume that $\theta(x;\bar m)$ is a conjunction of the form $$\bigwedge_{j} \phi_j^{\epsilon_j}(x;\bar m,b);\quad b\in A,~ \bar m\in J,$$ where $\epsilon_j$ is either 0 or $1$ depending on the position of the points in $\bar m$ with respect to the cuts $\mathfrak d(\phi_j;i)$. We can find an injection $\sigma : \bar m \rightarrow I$ such that:

-- for every $m_0$, $m_1$ in $\bar m$, if $m_0 <_J m_1$, then $\sigma(m_0)<_I \sigma(m_1)$;

-- for every index $j$ and $m_0\in \bar m$, the relative position of $\sigma(m_0)$ and the cut $\textsf {cut}_I(a,\phi_j;i)$ on $I$ is the same as that of $m_0$ and $\mathfrak d(\phi;i)$.

Then $\sigma$ is a partial isomorphism and $a\models \bigwedge_j \phi_j(x;\sigma(\bar m))$. Therefore $\theta(x;\bar m)$ is consistent and by compactness, $p(x)$ is consistent.
\end{proof}

\begin{cor}\label{sliding2}
Let $I$, $J$ be two dense sequences with no endpoints indiscernible over some set $A$ of same EM-type over $A$. Let $a$ and $b$ be tuples of the same length such that $\tp(a/I)$ and $\tp(b/J)$ are similar over $A$. Let $a'$ be any tuple. Then there is an indiscernible sequence $J'\supseteq J$ and a tuple $b'$ such that $\tp(bb'/J')$ is similar to $\tp(aa'/I)$ over $A$.
\end{cor}
\begin{proof}
By expanding, we can find a sequence $J'$ extending $J$ such that $\tp(b/J')$ is similar to $\tp(b/J)$ and the sequence $J'$ is indexed by a $|T|^+$-saturated dense linear order. It it then easy to find cuts $\mathfrak d(\phi;i)$ in $J'$ as in the previous lemma corresponding to the cuts $\textsf {cut}_I(aa',\phi;i)$ in a way compatible with the cuts $\textsf {cut}_{J'}(b,\phi;i)$ over $J'$. Lemma \ref{sliding1} gives us a tuple $b_0b_0'$ of same length as $aa'$ such that $\tp(b_0b_0'/J')$ is similar to $\tp(aa'/I)$. By assumption on the cuts $\mathfrak d(\phi;i)$, we have $\tp(b_0/J')=\tp(b/J')$ so by composing by an automorphism over $J'$, we obtain some $b'$ as required.
\end{proof}

\begin{cor}\label{sliding3}
Let $I$, $J$ be two dense sequences with no endpoints indiscernible over $A$ and of same EM-type over $A$. Let $a$ and $b$ be tuples of the same length such that $\tp(a/I)$ and $\tp(b/J)$ are similar over $A$. Let $I'\supseteq I$ be indiscernible and let $a'$ be any tuple. Then there is an indiscernible sequence $J'\supseteq J$ and a tuple $b'$ such that $\tp(bb'/J')$ is similar to $\tp(aa'/I')$ over $A$.
\end{cor}
\begin{proof}
Simply apply the previous corollary with $a'$ there equal to $a'\cup (I'\setminus I)$ here.
\end{proof}


\subsection{Weight and dp-minimality}

Let $(I_i)_{i<\alpha}$ be a family of indiscernible sequences and $A$ a set of parameters. We say that the sequences $(I_i)_{i<\alpha}$ are \emph{mutually indiscernible} over $A$ if for every $i<\alpha$, the sequence $I_i$ is indiscernible over $A\cup \{I_j : j<\alpha, j\neq i\}$.

The following observations are from \cite{Sh715}.

\begin{prop}\label{weight}
Let $(I_i)_{i<|T|^+}$ be mutually indiscernible sequences (over some set $A$) and let $d$ be a tuple of size at most $|T|$. Then there is some $i<|T|^+$ such that $I_i$ is indiscernible over $Ad$.
\end{prop}
\begin{proof}
Assume not, then for every $i<|T|^+$, we can find two tuples $\bar a_i$ and $\bar b_i$ of increasing elements from $I_i$ and a formula $\phi_i(x,\bar y)$ such that $d\models \phi_i(x,\bar a_i) \wedge \neg \phi_i(x,\bar b_i)$. Removing some sequences from the family, we may assume that $\phi_i=\phi$ does not depend on $i$. By mutual indiscernibility, we have $\tp(a_i/\{I_j : j\neq i\})=\tp(b_i/\{I_j:j\neq i\})$ for all $i<|T|^+$. It follows that for every $A\subseteq |T|^+$, we can find a tuple $d_A$ such that for all $i<|T|^+$, $d_A \models \phi(x,\bar a_i)$ if and only if $i\in A$. This contradicts NIP.
\end{proof}

\begin{cor}\label{weight2}
Let $M$ be some $\kappa$-saturated model, and let $(p_i)_{i<|T|^+}$ be a family of pairwise commuting invariant types over $M$. Let $p=\bigotimes_{i<|T|^+} p_i$ and $(a_i)_{i<|T|^+}\models p$. Let also $q\in S(M)$ be any type and $d\models q$. Then there is $i<|T|^+$ such that $(a_i,d)\models p_i \otimes q$.
\end{cor}
\begin{proof}
Build a Morley sequence $\llg (a_i^k)_{i<|T|^+} : 0<k<\omega \rrg$ of $p$ over everything and set $a_i^0=a_i$ for each $i$. Commutativity implies that the sequences $(a_{i}^k)_{k<\omega}$, $i<|T|^+$ are mutually indiscernible. The result then follows by Proposition \ref{weight}.
\end{proof}

Observe in particular that if $q$ is an invariant type, taking $b\models q|\{a_i : i<|T|^+\}$, we obtain that there is $i<|T|^+$ such that $p_i$ and $q$ commute.

\vspace{3pt} 
We will occasionally mention dp-minimal theories. They are theories for which the notion of weight suggested by Proposition \ref{weight} is equal to 1 on 1-types. This notion was introduced by Shelah in \cite{Sh863}.

\begin{defi}[Dp-minimal]
A theory $T$ is dp-minimal if it is $NIP$ and if for every indiscernible sequence $I$ and element $d$ of the home sort, there is a subdivision $I=I_1+I_2+I_3$ into convex sets, where $I_2$ is either reduced to a point or empty and $I_1$ and $I_3$ are both indiscernible over $d$.

Equivalently, for every two mutually indiscernible sequences $I$ and $J$ and element $d$, one of $I$ or $J$ is indiscernible over $d$.
\end{defi}

See \cite{dpmin} for the proof of the equivalence and \cite{dpbasic} for additional information.

Examples of dp-minimal theories include o-minimal and C-minimal theories and the p-adics.

\section{Distal theories}

\subsection{Indiscernible sequences}

We now state the main definition of this paper.

\begin{defi}[Distal]
An indiscernible sequence $I$ is distal if for any dense sequence $J$ of same EM-type as $I$, and any distinct Dedekind cuts $\mathfrak c_1$ and $\mathfrak c_2$ of $J$, if $a$ fills $\mathfrak c_1$ and $b$ fills $\mathfrak c_2$, then $a \downfree_J b$.

An $NIP$ theory $T$ is \emph{distal} if all indiscernible sequences are distal.
\end{defi}

\begin{rem}
Equivalently the two types $\lim(\mathfrak c_1/J)$ and $\lim(\mathfrak c_2/J)$ are weakly orthogonal.
\end{rem}

\begin{lemme}\label{distalsimpl}
If $I$ is dense and has two distinct Dedekind cuts $\mathfrak c_1$ and $\mathfrak c_2$, then it is distal if and only if $\lim(\mathfrak c_1/I)$ and $\lim(\mathfrak c_2/I)$ are weakly orthogonal ({\it i.e.}, there is no need for $J$ in the definition).
\end{lemme}
\begin{proof}
Left to right is obvious. We show the converse. If $I$ is not distal, then there is some dense sequence $J$ of the same EM-type, two distinct Dedekind cuts $\mathfrak d_1$ and $\mathfrak d_2$ of $J$, some $a_1$ filling $\mathfrak d_1$ and $a_2$ filling $\mathfrak d_2$ such that $a_1 \ndownfree_J a_2$. Let $\phi(a_1,a_2,\bar m)$ be a formula witnessing that, with $\bar m \in I$. Take a countable $J'\subseteq J$ containing $\bar m$ such that $a_1$ and $a_2$ fill Dedekind cuts of $J'$. Replacing $J$ by $J'$, we may assume that $J$ is countable.

Then by expanding, we can find some $J_0 \supseteq J$ and an automorphism $\sigma$ mapping $J_0$ onto $I$ and such that the cut $\mathfrak d_1$ (resp. $\mathfrak d_2$) is mapped to $\mathfrak c_1$ (resp. $\mathfrak c_2$) and the types $\tp(a_1,a_2/J)$ and $\tp(a_1,a_2/J_0)$ are similar. Then, the points $\sigma(a_1)$ and $\sigma(a_2)$ fill respectively the cuts $\mathfrak c_1$ and $\mathfrak c_2$ and $\phi(\sigma(a_1),\sigma(a_2),\sigma(\bar m))$ holds. Therefore $\sigma(a_1) \ndownfree_{I} \sigma(a_2)$ and it follows that the two limit types $\lim(\mathfrak c_1/I)$ and $\lim(\mathfrak c_2/I)$ are not weakly orthogonal.
\end{proof}

Actually, it will follow from Lemma \ref{defitrans} that the hypothesis that $I$ is dense can be removed.

\begin{ex}
Assume $I$ is an indiscernible sequence, $f$ a definable function such that $f(I)$ is totally indiscernible (non constant), then $I$ is not distal. To see this, take $a$ and $b$ in the definition such that $f(a) = f(b)$. See \ref{orth} for a more general result.
\end{ex}

\begin{ex}
In DLO, any two 1-types concentrating on different cuts are weakly orthogonal. It is easy then to check that it is a distal theory. We will see (Corollary \ref{omin}) that in fact any o-minimal theory is distal.
\end{ex}

\begin{lemme}\label{twoton}
Assume $I$ is a dense indiscernible distal sequence, and $\mathfrak c_0,...,\mathfrak c_{n-1}$ are pairwise distinct Dedekind cuts. If for each $i<n$, $a_i$ fills $\mathfrak c_i$ then the family $(a_i)_{i<n}$ is $I$-independent.
\end{lemme}
\begin{proof}
We prove it by induction on $n$. for $n=2$, it is Lemma \ref{distalsimpl}. Assume it holds for $n$ and consider a family $(\mathfrak c_i)_{i<n+1}$ and $(a_i)_{i<n+1}$ as in the hypothesis. Let $I'=I \cup \{a_0\}$ (where $a_0$ is inserted in the cut $\mathfrak c_0$). Each cut $\mathfrak c_i$ naturally induces a cut $\mathfrak c'_i$ of $I'$. By the case $n=2$, for each $0<i<n+1$, $a_i$ fills $\mathfrak c_i'$. The sequence $I'$ is also distal, so by induction $(a_i)_{0<i<n+1}$ is $I'$-independent. Therefore $(a_i)_{i<n+1}$ is $I$-independent.
\end{proof}

\begin{lemme}[External characterization of distality]\label{defitrans}
A sequence $I$ is distal if and only if the following property holds:
For every set $A$, tuple $b$ and $A$-indiscernible sequence $I'=I_1+I_2$ ($I_1$ and $I_2$ without endpoints, EM-$\tp(I')$=EM-$\tp(I)$), if $I_1+b+I_2$ is indiscernible, it is $A$-indiscernible.
\end{lemme}
\begin{proof}
Assume that $I$ is distal, but the conclusion does not hold. Then there is some $I'=I_1+I_2$ and formula $\phi(x)$ with parameters from $A\cup I_1 \cup I_2$ which witnesses it. This means $\phi(b)$ holds and there is $(I_1',I_2') \unlhd (I_1,I_2)$ such that $\neg \phi(a)$ holds for $a\in I_1'\cup I_2'$. Restricting even more if necessary, we may assume that $I_1'+I_2'$ is indiscernible over the parameters of $\phi$. So replacing $I'$ by that latter sequence, we may assume that all the parameters are from $A$. Then, we may freely enlarge $I'$, so assume that it is dense.

As $I'$ is $A$-indiscernible, for every cut $\mathfrak c$ of $I'$, there is $b'$ filling it such that $\phi(b')$ holds. Fix an increasing sequence $(\mathfrak c_k)_{k<\omega}$ of such cuts. For every $k<\omega$, let $b_k$ fill $\mathfrak c_k$ such that $\phi(b_k)$ holds.
The sequence $I'$ is distal (because $I'$ and $I$ have same EM-type) so by Lemma \ref{twoton}, the sequence formed by adding all those points to $I'$ is still indiscernible. Therefore $\phi(x)$ has infinite alternation number, contradicting $NIP$.

The converse is easy.
\end{proof}

The following technical lemma will be used repeatedly.

\begin{lemme}[Strong base change]\label{limittype}
Let $I$ be an indiscernible sequence and $A \supseteq I$ a set of parameters. Let $(\mathfrak c_i)_{i<\alpha}$ be a sequence of pairwise distinct polarized Dedekind cuts in $I$. For each $i<\alpha$ let $d_i$ fill the cut $\mathfrak c_i$. Then there exist $(d_i')_{i<\alpha}$ such that $\tp((d_i')_{i<\alpha}/I)=\tp((d_i)_{i<\alpha}/I)$ and for each $i<\alpha$, $\tp(d_i'/A) = \lim(\mathfrak c_i/A)$.
\end{lemme}
\begin{proof}

Assume the result does not hold. Then by compactness, we may assume that $\alpha=n$ is finite and that there is a formula $\phi(x_0,..,x_{n-1}) \in tp((d_i)_{i<n}/I)$ and formulas $\psi_i(x_i) \in \lim(\mathfrak c_i/m)$ for some finite $m \in A^k$ such that $\phi(x_0,..,x_{n-1})\wedge \bigwedge_i \psi_i(x_i)$ is inconsistent. Let $I_0$ denote the parameters of $\phi$, and assume $I_0 \subseteq m$.  

Assume for simplicity that $n=2$ (the proof for $n>2$ is the same) and without loss each $\mathfrak c_i$ is polarized as $\mathfrak c_i^-$. For $i=0,1$  take $(J_i,J_i') \unlhd \mathfrak c_i$ such that $\psi_i$ holds on all elements of $J_i$ and $J_i \cup J_i'$ contains no element of $I_0$. Then $J_0+J_0'$ and $J_1+J_1'$ are mutually indiscernible over $I_0$. So for every two cuts $\mathfrak d_0$ and $\mathfrak d_1$ respectively from $J_0+J_0'$ and $J_1+J_1'$, we can find points $e_0$ and $e_1$ filling those cuts (even seen as cuts of $I$) such that $\phi(e_0,e_1)$ holds.

Take two cuts $\mathfrak d_0$ and $\mathfrak d_1$ of $I$ such that they are respectively interior to $J_0$ and $J_1$. Fill $\mathfrak d_0$ by $e_0$ and $\mathfrak d_1$ by $e_1$ such that $\phi(e_0,e_1)$ holds. By hypothesis, either $\neg \psi_0(e_0)$ or $\neg \psi_1(e_1)$ holds. Assume $\neg \psi_1(e_1)$ holds. Now forget about $e_0$ and set $I'=I \cup \{e_1\}$. Then $I'$ is indiscernible and we take it as our new $I$. Set $J_0' = J_0$ and let $J_1'$ be an initial segment of $J_1$ not containing $\mathfrak d_1$ and make the same construction. We obtain new points $(e_0^1,e_1^1)$ that fill the cuts $\mathfrak d_0^1,\mathfrak d_1^1$ of $J_0'$ and $J_1'$ such that $\neg \psi_0(e_0^1) \vee \neg \psi_1(e_1^1)$ holds. Without loss (as we will iterate infinitely many times) again $\neg \psi_1(e_1^1)$ holds.

Iterate this $\omega$ time to obtain a sequence of points $e_1^k$ and cuts $\mathfrak d_1^k$ in $J_1$ such that $I$ with all the points $e_1^k$ added in the cuts $\mathfrak d_1^k$ is indiscernible and $\neg \psi_1(e_1^k)$ holds for all $n$. But $\psi_1(x)$ holds for all $x \in J_1$ so $\psi_1$ has infinite alternation rank, contradicting $NIP$.
\end{proof}

\begin{cor}[Base change]\label{basechange}
The notion of being distal is stable both ways under base change: If $I$ is $A$-indiscernible, then $I$ is distal in $T(A)$ if and only if it is distal in $T$.
\end{cor}
\begin{proof}
Assume $I$ is distal in $T$. Notice that the property stated in Lemma \ref{defitrans} is preserved under naming parameters (because we can incorporate them in the set $A$). This implies that $I$ is distal in $T(A)$.

Conversely, assume $I$ is not distal in $T$. Increase $I$ to some large $A$-indiscernible sequence $J_1+J_2+J_3$ and take $a,b$ such that $J_1+a+J_2+J_3$ and $J_1+J_2+b+J_3$ are indiscernible, but $J_1+a+J_2+b+J_3$ is not. By strong base change, we may assume that $a$ and $b$ realize the limit types over $A$ of the cuts they define. Then $J_1+a+J_2+J_3$ and $J_1+J_2+b+J_3$ are $A$-indiscernible, giving a counter-example to distality in $T(A)$.
\end{proof}

\begin{lemme}\label{dpminlemme}
If $T$ is dp-minimal and $I$ is an indiscernible sequence of elements of the home sort which is not totally-indiscernible, then $I$ is distal.
\end{lemme}
\begin{proof}
Write $I=(d_i)_{i\in \mathcal I}$ and assume that it is not totally indiscernible. Working over some base $A$ if necessary, we may assume that there is a formula $\phi(x,y)\in L(A)$ which orders the sequence $I$ and such that $I$ is indiscernible over $A$. So we have $\phi(d_i,d_j) \iff i<j$. (Extend the sequence $I$ to some $J_1+I+J_2$ and take $A=J_1+J_2$.)

Without loss $\mathcal I$ is a dense order and can be written as $\mathcal I_1+\mathcal I_2+\mathcal I_3$, the three pieces being infinite without end points. Write $I=I_1+I_2+I_3$ in the obvious way. Let $a$ fill the cut $\mathfrak c_a=(I_1,I_2+I_3)$ and $b$ fill $\mathfrak c_b=(I_1+I_2,I_3)$. Assume that $a$ and $b$ contradict distality of $I$. So there is a formula $\psi(x,y)\in L(AI)$ such that $\psi(a,b)$ holds and witnesses $a \ndownfree_I b$. Let $\bar d=(d_{i_1},...,d_{i_n})$ be the parameters of $\psi$ coming from $I$ with $i_1<\ldots<i_n$. Let $s$ be such that exactly $i_1,...,i_s$ are from $\mathcal I_1$ and $t$ such that exactly $i_{s+1},...,i_t$ are from $\mathcal I_2$. Let $\mathcal I_1'$ be an end segment of $\mathcal I_1$ above $i_s$ and $\mathcal I_3'$ an initial segment of $\mathcal I_3$ below $i_{t+1}$.

Let $\bar d_1=(d_{i_1},...,d_{i_s})$ and $\bar d_3=(d_{i_{t+1}},...,d_{i_n})$. Consider the sequence $J= \llg d_i\hat{~}\bar d_1\hat{~}\bar d_3 : i\in \mathcal I_1'+\mathcal I_2\rrg + \llg b\hat{~}\bar d_1\hat{~}\bar d_3\rrg +  \llg d_i\hat{~}\bar d_1\hat{~}\bar d_3 : i\in \mathcal I_3'\rrg$. It is an indiscernible sequence. By dp-minimality applied to $J$ and $a$, we know that $J$ breaks into $J_1+J_2+J_3$, $J_2$  having at most one element, and such that $J_1$ and $J_3$ are indiscernible over $a$. Considering the formula $\phi(x,a)$, we know that $J_1$ must be equal to $\llg d_i\hat{~}\bar d_1\hat{~}\bar d_3 : i\in \mathcal I_1'\rrg$. And then $J_2$ is empty and $J_3$ is the rest of the sequence. In particular the tuple $b\hat{~}\bar d_1\hat{~}\bar d_3$ lies inside $J_3$ as do all the parameters of $\psi(x,y)$. As $\psi(a,b)$ holds but $\neg \psi(a,d_i)$ holds for $i\in \mathcal I_3$, we get a contradiction to the indiscernibility of $J_3$ over $a$.
\end{proof}

\begin{lemme}
Let $T$ be distal, $I$ and $J$ are two mutually indiscernible sequences. Let $\mathfrak c$ (resp. $\mathfrak d$) be a cut in the interior of $I$ (resp. $J$). Then $\lim(\mathfrak c/IJ)$ and $\lim(\mathfrak d/IJ)$ are weakly orthogonal.
\end{lemme}
\begin{proof}
Write $I=(a_i)_{i\in \mathcal I}$ and $J=(b_j)_{j\in \mathcal J}$. Assume the conclusion does not hold. Then there are $a\models \lim(\mathfrak c/IJ)$ and $b\models \lim(\mathfrak d/IJ)$ and a formula $\phi(x,y)\in L(IJ)$ such that $\phi(a,b)$ holds, but $\lim(\mathfrak c)\otimes \lim(\mathfrak d) \vdash \neg \phi(x,y)$. Let $\mathcal K$ be a countable dense linear order without end points. Pick embedding $\tau_1: \mathcal K \rightarrow \mathcal I$ and $\tau_2:\mathcal K \rightarrow \mathcal J$ such that:\\
-- $\mathfrak c$ induces a Dedekind cut on $\tau_1(\mathcal K)$ and $\mathcal d$ induces a Dedekind cut on $\tau_2(\mathcal K)$;\\
-- identifying $\tau_1(\mathcal K)$ and $\tau_2(\mathcal K)$, those two Dedekind cuts are distinct;\\
-- the parameters of $\phi(x,y)$ belong to $\{a_i : i\in \tau_1(\mathcal K)\}\cup \{b_j : j\in \tau_2(\mathcal K)\}$.

Let $K$ be the sequence $\langle a_{\tau_1(t)}\hat{~}b_{\tau_2(t)} : t\in \mathcal K\rangle$. Let $\mathfrak c'$ and $\mathfrak d'$ denote the two cuts naturally induced by $\mathfrak c$ and $\mathfrak d$ on $K$. There are tuples $b_*$ and $a_*$ such that $a\hat{~}b_*$ fills $\mathfrak c'$ and $a_*\hat{~}b$ fill $\mathfrak d'$. By distality of $K$, $a\hat{~}b_* \downfree_K a_*\hat{~}b$ and $\phi(a,b)$ holds. This contradicts the assumption.
\end{proof}

\begin{defi}[Weakly linked]\label{weaklylinked}
Let $\llg (a_i,b_i): i\in \mathcal I \rrg$ be an indiscernible sequence of pairs. We say that $(a_i)_{i\in \mathcal I}$ and $( b_i)_{i\in \mathcal I}$ are weakly linked if for all disjoint subsets $\mathcal I_1$ and $\mathcal I_2$ of $\mathcal I$, $( a_i)_{i\in\mathcal I_1}$ and $( b_i)_{i\in\mathcal I_2}$ are mutually indiscernible.
\end{defi}

\begin{obs}\label{basicobs}
\begin{enumerate}
\item If $\llg (a_i,b_i):i\in \mathcal I\rrg$ is $A$-indiscernible and $( a_i)_{i\in \mathcal I}$ and $( b_i)_{i \in \mathcal I}$ are mutually indiscernible, then they are mutually indiscernible over $A$.
\item If $\llg (a_i,b_i): i\in \mathcal I\rrg$ is $A$-indiscernible and $( a_i)_{i\in \mathcal I}$ and $( b_i)_{i\in \mathcal I}$ are weakly linked, then they are weakly linked over $A$.
\end{enumerate}
\end{obs}

\begin{lemme}\label{declemma}
Let $\llg (a_i,b_i):i\in \mathcal I\rrg$ be indiscernible.
\begin{enumerate}
\item If $( a_i)_{i\in \mathcal I}$ and $( b_i)_{i\in \mathcal I}$ are weakly linked and $( a_i)_{i\in \mathcal I}$ is distal, then $( a_i)_{i\in \mathcal I}$ and $( b_i)_{i\in \mathcal I}$ are mutually indiscernible.
\item If $( b_i)_{i\in \mathcal I}$ is totally indiscernible, then $( a_i)_{i\in \mathcal I}$ and $( b_i)_{i\in \mathcal I}$ are weakly linked.
\end{enumerate}
\end{lemme}
\begin{proof}
(1). Without loss, we may assume that $\mathcal I$ is dense. Pick some finite $\mathcal I_2 \subset \mathcal I$. Then $( a_i)_{i \notin \mathcal I_2}$ is indiscernible over $B=( b_i)_{i\in \mathcal I_2}$. By applying repeatedly Lemma \ref{defitrans}, we obtain that $( a_i)_{i \in \mathcal I}$ is indiscernible over $B$. This is enough.

(2). Assume $\mathcal I$ is dense and big enough, take $\mathcal I_1 \subset \mathcal I$ finite and let $A=( a_i)_{i\in \mathcal I_1}$. By shrinking of indiscernibles and using total indiscernibility of $( b_i)_{i\in \mathcal I}$, there is $\mathcal I_2\subset \mathcal I$ of size at most $|T|$ such that $(b_i)_{i\in \mathcal I \setminus\mathcal I_2}$ is indiscernible over $A$. By indiscernibility of $\llg (a_i,b_i):i\in \mathcal I \rrg$, we may take $\mathcal I_2=\mathcal I_1$. Therefore $( a_i)_{i\in \mathcal I}$ and $( b_i)_{i\in \mathcal I}$ are weakly linked.
\end{proof}

\begin{cor}\label{orth}
Let $\llg (a_i,b_i): i\in \mathcal I\rrg$ be an indiscernible sequence. Assume $( a_i)_{i\in \mathcal I}$ is totally indiscernible and $( b_i)_{i\in \mathcal I}$ is distal, then $( a_i)_{i\in \mathcal I}$ and $( b_i)_{i\in \mathcal I}$ are mutually indiscernible.
\end{cor}

\subsection{Invariant types}

We prove here a characterization of distality in terms of invariant types.

If $M$ is a $\kappa$-saturated model, by an invariant type over $M$, we mean a type $p\in S(M)$ invariant over some $A\subset M$, $|A|<\kappa$. If $p$ and $q$ are two invariant types over $M$, then we can define the products $p_x\otimes q_y$ and $q_y\otimes p_x$ as explained in the introduction. The types $p$ and $q$ \emph{commute} if those two products are equal.

\begin{lemme}\label{commute}
Assume $T$ is distal. Let $M$ be $\kappa$-saturated and let $p,q\in S(M)$ be invariant types. If $p_x \otimes q_y = q_y \otimes p_x$, then $p$ and $q$ are orthogonal.
\end{lemme}
\begin{proof}
Let $b\models q$ and let $N \prec M$ a model of size $<\kappa$ such that $p$ and $q$ are $N$-invariant. Let $I\subset M$ be a Morley sequence of $p$ over $N$. Let $a$ realize $p$, and build $I'$ a Morley sequence of $p$ over $Mab$. The hypothesis implies that $p^{(\omega)}$ and $q$ commute (as $\otimes$ is associative). Thus $b\models q|_{MI'}$ and in particular, $I+I'$ is indiscernible over $Nb$. By distality, $I+a+I'$ is also $Nb$-indiscernible. This proves that $tp(a,b/N)$ is determined.

As this is true for any small $N$ over which $p$ and $q$ are invariant, the types $p$ and $q$ are orthogonal.
\end{proof}

\begin{prop}
The theory $T$ is distal if and only if any two global invariant types $p$ and $q$ that commute are orthogonal.
\end{prop}
\begin{proof}
Lemma \ref{commute} gives one implication. Conversely, assume that $T$ is not distal. Then there is a dense indiscernible sequence $I$, two distinct Dedekind cuts $\mathfrak c_1$ and $\mathfrak c_2$ and $a$ and $b$ filling them such that $a \ndownfree_I b$. By Lemma \ref{limittype} (strong base change), we may assume that $I \subset M$, for $M$ a large saturated model, and $a\models \lim(\mathfrak c_1^-/M)$, $b\models \lim(\mathfrak c_2^-/M)$. Then the types $p=\lim(\mathfrak c_1^{-}/M)$ and $q=\lim(\mathfrak c_2^-/M)$ have the required property.
\end{proof}

Consider $p,q\in S(M)$ and assume only that $p$ is invariant. Then $p_x\otimes q_y$ is well defined, but $q_y \otimes p_x$ does not make sense a priori. We show now how to define $q_y\otimes p_x$.

Let $M$ be $\kappa$-saturated and $p\in S(M)$ an $A$-invariant type for some $A\subset M$ of size $<\kappa$. We define an $M$-invariant type $p'\in S(\monster)$ as follows: Fix a formula $\phi(x;b)\in L(\monster)$ and a maximal Morley sequence $(a_1,\ldots,a_n)$ of $p$ over $A$ such that $\neg (\phi(a_i;b)\leftrightarrow \phi(a_{i+1};b))$ holds for all $i<n$ and each $a_i$ is in $M$. Set $\phi(x;b)\in p'$ if and only if $\models \phi(a_n;b)$. We will call $p'$ the \emph{inverse} of $p$ over $M$.

Now if $q_y\in S(M)$ is any type, then we define $q_y \otimes p_x$ to the be $p'_x \otimes q_y\in S(M)$. Notice that if $q$ was invariant to begin with, then the two definitions of $q_y \otimes p_x$ coincide. Note also that the associativity relation: $p_x \otimes (q_y \otimes r_z)=(p_x \otimes q_y) \otimes r_z$ holds in all possible cases (each product is well defined if and only if at least two of $p,q,r$ are invariant).

\smallskip
The following generalizes Lemma \ref{commute}, the proof is the same, using Lemma \ref{stationnaire} to build the Morley sequence $I$ of $p$ inside $M$.

\begin{lemme}\label{commute2}
Assume $T$ is distal. Let $M$ be $\kappa$-saturated ($\kappa\geq |T|^+$), $p\in S(M)$ be $A$-invariant for some $A$ of size $<\kappa$ and $q\in S(M)$ be any type. If $p_x \otimes q_y = q_y \otimes p_x$, then $p$ and $q$ are orthogonal.
\end{lemme}

We record the following lemma for future needs.

\begin{lemme}\label{stationnaire}
Let $M$ be $\kappa$-saturated, $\kappa\geq |T|^+$. Let $p,q \in S(M)$, $p$ being $A$-invariant for some $|A|<\kappa$. Then there is some $B\subset M$, $|B|<\kappa$, such that $A \subseteq B$ and for $b\models q$ and any $a,a' \in M$ such that $a,a'\models p|_B$, we have $\tp(a,b/A)=\tp(a',b/A)$.
\end{lemme}
\begin{proof}
Fix a formula $\phi(x,y;c)\in L(A)$ and take $(a_1,\ldots,a_n)$ in $M$ a maximal Morley sequence of $p$ over $A$ such that $\neg (\phi(a_i,b;c)\leftrightarrow \phi(a_{i+1},b;c))$ holds for all $i<n$. Then for each $a\in M$, $a\models p|_{Aa_1..a_n}$ we have $\models \phi(a,b;c) \leftrightarrow \phi(a_n,b;c)$.

Take $B$ to contain all the $a_i$'s obtain by letting $\phi(x,y;c)$ range in $L(A)$.
%
%
%
%
\end{proof}

\subsection{Generically stable measures}

We prove in this section that distal theories are exactly those theories in which generically stable measures are smooth. We consider this as a justification that distality is a meaningful notion. It was proved in \cite{Finding} that o-minimal theories and the p-adics have this property. This latter result will be generalized in the next section, where we prove that distality can be checked in dimension 1.

We have two tools at our disposal to link indiscernible sequences of tuples to measures. In one direction, starting with an indiscernible sequence of tuples, we can form the average measure. This construction is defined in \cite{NIP3}, extended in \cite{Finding} and recalled below. In the opposite direction, starting with a generically stable measure $\mu$ (or in fact any invariant measure), we can consider the product $\mu^{(\omega)}$ in variables $x_1,x_2,\ldots$. We then want to realize it in some way. We do this by taking smooth extensions; see the proof of Proposition \ref{gentrans}.
\\

Let $I=(a_t)_{t\in [0,1]}$ be an indiscernible sequence. We can define the \emph{average measure} $\mu$ of $I$ as the global measure defined by $\mu(\phi(x))=\lambda_0(\{t\in [0,1]: a_t \models \phi(x)\})$, where $\lambda_0$ is the Lebesgue measure. That measure is generically stable (in fact definable and finitely satisfiable over $I$).

The \emph{support} of a measure $\mu \in \mathcal M(A)$ is the set of weakly-random types for $\mu$, namely the set of types $p\in S(A)$ such that $p\vdash \neg \phi(x)$ for every formula $\phi(x)\in L(A)$ such that $\mu(\phi(x))=0$. We will denote it by $S(\mu)$.

\begin{lemme}
Let $\mu$ be the average measure of the indiscernible sequence $I=(a_t)_{t\in [0,1]}$ (over $\monster$). Then the support $S(\mu)$ of $\mu$ is exactly the set of limit types of cuts of $I$.
\end{lemme}
\begin{proof}
First, if $\phi(x)$ is satisfied by some $\lim(\mathfrak c)$, $\mathfrak c$ a cut in $I$, then $\phi(x)$ holds on a subsequence, cofinal in $\mathfrak c$, and therefore has positive measure. Conversely, let $p(x)\in S(\mu)$ (a global type). For each $\phi(x)\in p$, the set $\{t\in[0,1] : \models \phi(a_t)\}$ is infinite. By compactness of $[0,1]$, there is $r\in [0,1]$ which is in the closure of all of those sets as $\phi(x)$ varies in $L(\monster)$. If $I$ is totally indiscernible, then $\mu$ is the unique limit type of $I$, so assume that this is not the case. Then $I$ is ordered by some formula $\psi(x,y)\in L(\monster)$. The type $p$ must satisfy either $\psi(x,a_r)$ or $\psi(a_r,x)$. In the first case, $p$ is equal to the limit type to the left of $a_r$ and in the second case, to the limit type to the right of $a_r$.

%
\end{proof}

\begin{prop}[Smooth measures imply distality]
Let $I$ be an indiscernible sequence indexed by $[0,1]$, and $\mu$ be the average measure of $I$ over some model $M$. Then $\mu$ is smooth if and only if $I$ is distal.
\end{prop}
\begin{proof}
Assume $\mu$ is not smooth and $I$ is distal. Then there exists a formula $\phi(x,a)\in L(\monster)$ such that the set of $p\in S(M)$ such that $p$ neither implies $\phi(x,a)$ nor its negation has positive measure (in other words, $p\in \partial\phi$). We know that the support of $\mu$ is exactly the limit types of cuts in $I$. Therefore, one can find $\omega$ such cuts $(\mathfrak c_i)_{i<\omega}$ in $\partial\phi$. Remove countably many points from $I$ (thus not affecting any limit types) so that the cuts $\mathfrak c_i$ become Dedekind.

Restricting to some sub-interval of $[0,1]$, we may assume that $\phi(x,a)$ has constant truth value on $I$. Without loss, it holds on all members of $I$. For each index $i$, as $\lim(\mathfrak c_i)\in \partial \phi$, there is $b_i$ filling the cut $\mathfrak c_i$ over $I$ such that $\neg \phi(b_i,a)$ holds. As $I$ is distal, the sequence formed by adding all the $b_i$ to $I$ is still indiscernible. But then the formula $\phi(x,a)$ has infinite alternation number.

Conversely, assume that $I$ is not distal. If $J$ is an indiscernible sequence, we write $J'$ for the sequence $J$ with the endpoints removed. We can find a partition $I=I_1+I_2+I_3$ and points $b_1, b_2$ such that $I_1'+b_1+I_2'+I_3'$ and $I_1'+I_2'+b_2+I_3'$ are indiscernible, but $I_1'+b_1+I_2'+b_2+I_3'$ is not. Without loss, assume that $I_1$ and $I_2$ have no last element. By strong base change, we may assume that the types of $b_1$ and $b_2$ over $M$ are respectively $\lim(I_1)$ and $\lim(I_2)$. There is a formula $\phi$, parameters $i_k \subset I_k$ and $b_1'$ realizing the same type as $b_1$ over $M$ such that $\phi(i_1,b_1,i_2,b_2,i_3) \wedge \neg \phi(i_1,b_1',i_2,b_2,i_3)$ holds. Then the border $\partial \phi$ of $\phi(i_1,x,i_2,b_2,i_3)$ contains all limit types of cuts between $i_1$ and $i_2$ and has non zero measure. This proves that $\mu$ is not smooth.
\end{proof}

\begin{cor}\label{cor}
If all generically stable measures are smooth, then $T$ is distal.
\end{cor}

Before proving the converse, we generalize some earlier lemmas from types to measures. Recall the following fact (which follows for example from Proposition 3.3 of \cite{NIP1}).

\begin{fait}\label{fact1}
Let $\mu_x$ be a measure and $(a_i)_{i\in \mathcal I}$ an indiscernible sequence. Let $\phi(x;y)$ be a formula and $\epsilon>0$. Then, for some $N$, there do not exist $i_1<\cdots<i_N$ such that $|\mu(\phi(x;a_{i_k}))-\mu(\phi(x;a_{i_{k+1}}))|\geq \epsilon$ for all $j=1\ldots N-1$.
\end{fait}

By a measure $\mu_{x_1,x_2,...}$ being indiscernible, we mean that for any formula $\phi(x_1,\ldots,x_n)$ and any increasing map $\tau:\omega \to \omega$, we have $\mu(\phi(x_1,\ldots,x_n))=\mu(\phi(x_{\tau(1)},\ldots,x_{\tau(n)}))$. We now state the analogue of the previous fact with an indiscernible sequence of measures, which is Corollary 2.12 of \cite{NIP3}. 

\begin{fait}\label{fact2}
Let $\mu_{x_1,x_2,...}$ be indiscernible, and $b$ any tuple. Fix some formula $\phi(x;y)$ and $\epsilon>0$, then for some $N$, there do not exist $i_1<\cdots<i_N$ such that $|\mu(\phi(x_{i_j};b))-\mu(\phi(x_{i_{j+1}};b))|\geq \epsilon$ for $j=1\ldots N-1$.
\end{fait}

In particular, if $\mu_{x_1,x_2,...}$ is totally indiscernible, {\it i.e.}, remains indiscernible when we permute the variables, then given $\phi(x;y)$, $\epsilon>0$ and $b$, for some $N$ there do not exist $i_1,\ldots,i_N$  such that $|\mu(\phi(x_{i_j};b))-\mu(\phi(x_{i_{j+1}};b))|\geq \epsilon$ for $j=1\ldots N-1$.

If $I=(a_i)_{i\in \mathcal I}$ and $\mu_x$ is a measure over $\{a_i:i\in \mathcal I\}$, we say that $I$ is \emph{$\mu$-indiscernible} if for all $\phi(x;y_1,\ldots,y_n)$, for all $t_1<\ldots <t_n$ and $s_1<\ldots<s_n$ in $\mathcal I$ we have $\mu(\phi(x;a_{t_1},\ldots,a_{t_n}))=\mu(\phi(x;a_{s_1},\ldots,a_{s_n}))$.

\begin{lemme}\label{lem_meas1}
Let $I_1+a+I_2$ be an indiscernible distal sequence, $I_1$ and $I_2$ without endpoints. If $\mu_x$ is a measure such that $I_1+I_2$ is $\mu$-indiscernible, then $I_1+a+I_2$ is also $\mu$-indiscernible.
\end{lemme}
\begin{proof}
The proof is the same as that of Lemma \ref{defitrans} using Fact \ref{fact1}.
\end{proof}

\begin{lemme}\label{lem_meas2}
Let $\mu_{x_1,x_2,...}$ be totally indiscernible and let $(b_i)_{i<\omega}$ be a distal indiscernible sequence. Assume that the measure $\eta_{(x_1,y_1),(x_2,y_2),...}$ is indiscernible, where $\eta$ is defined by $\eta(\phi(x_1,x_2,...;y_1,y_2,...))=\mu(\phi(x_1,x_2,...;b_1,b_2,...))$. Then the sequence $(b_i)_{i<\omega}$ is $\mu$-indiscernible.
\end{lemme}
\begin{proof}
This is the analogue of Corollary \ref{orth}. The same proof goes through. Namely, we first use indiscernibility to increase the index set from $\omega$ to a dense order $\mathcal I$. Next, let $\phi(x_{t_1},...,x_{t_n},b_{t_1},...,b_{t_n};x_s)$ be a formula, $t_1,\ldots,t_n\in \mathcal I$ are fixed and $\mathcal J\subset \mathcal I$ is disjoint for those points. Then using the remark following Fact \ref{fact2} and the indiscernibility of the $\eta$, we show that $\eta(\phi(x_{t_1},...,b_{t_1},...;x_s))$ is constant as $s$ varies in $\mathcal J$. From this, we conclude that the sequences are weakly linked, namely for any $\mathcal I_1,\mathcal I_2$ disjoint subsets of $\mathcal I$, the sequence $\llg b_i:i\in \mathcal I_1\rrg$ is $\mu'$-indiscernible, where $\mu'$ is the restriction of $\mu$ to the variables $(x_i:i\in \mathcal I_2)$.

Finally, we show exactly as in Lemma \ref{declemma} (1), that the sequence $(b_i)_{i<\omega}$ is $\mu$-indiscernible.
\end{proof}

\begin{prop}\label{gentrans}
If $T$ is distal, then all generically stable measures are smooth.
\end{prop}
\begin{proof}
Assume that $T$ is distal and take $\mu$ a generically stable measure over some $|T|^+$-saturated model $N$. The unique global invariant extension of it will also be denoted by $\mu$. Let $a$ be a tuple. We will show that $\mu$ and $\tp(a/N)$ are weakly orthogonal.

Let $\mu'$ be an extension of $\mu$ to $Na$. Take a smooth extension $\mu''$ of $\mu'$ to some $B\supseteq Na$. Let $(B_i)_{i<\omega}$ be a coheir sequence in $\tp(B/N)$, with $B_0=B$. The measure $\mu$ is definable over $B$, and for each $i<\omega$, we can consider the measure $\mu^i_{x_i}$ which is defined over $B_i$ the same way $\mu$ is defined over $B$ (using the canonical bijection from $B$ to $B_i$).

Consider the measure $\lambda_{\llg x_i, i<\omega\rrg}$ defined as $\bigotimes_{i<\omega} \mu^i_{x_i}$ (this does not depend on the order of the factors since the $\mu^i$'s are generically stable).

\vspace{4pt}
\noindent
\un{Claim}: The measure $\lambda_{x_0,x_1,...}$ is totally indiscernible over $N$, namely for every formula $\phi(x_0,\ldots,x_{n-1})\in L(N)$ and any permutation $\tau$ of $\omega$, we have $\lambda(\phi(x_0,\ldots,x_{n-1}))=\lambda(\phi(x_{\tau(0)},\ldots,x_{\tau(n-1)}))$.

\vspace{2pt}
{\it Proof:} Note that $\tp(B_1/B_0 N)$ is non-forking over $N$. In particular $\mu^1_{x_1}|_{B_0 N}$ does not fork over $N$ (as it is finitely satisfiable in $B_1$) so by Proposition 3.3 of \cite{NIP3} $\mu^1|_{B_0N} = \mu|_{B_0N}$. This implies, as $\mu^0$ is invariant over $B_0$, that $\mu^0_{x_0} \otimes \mu^1_{x_1}|_{B_0N} = \mu^0_{x_0} \otimes \mu_{x_1}|_{B_0N}$. On the other hand, $\mu^0_{x_0}\otimes \mu_{x_1}=\mu_{x_1}\otimes \mu^0_{x_0}$ and as $\mu^0|_N=\mu|_N$, we have $\mu_{x_1}\otimes \mu^0_{x_0}|_N=\mu_{x_1}\otimes \mu_{x_0}|_N$. Putting it all together, we obtain $\mu^0_{x_0} \otimes \mu^1_{x_1}|_N=\mu_{x_0} \otimes \mu_{x_1}|_N$.

Iterating this we get, $\lambda|_N = \mu^{(\omega)}|_N$. As $\mu$ is generically stable, $\lambda_{x_0,x_1,...}$ is totally indiscernible over $N$.

\vspace{4pt}
Now define a measure $\eta_{(x_0,y_0),(x_1,y_1)...}$ over $N$, where $y_i$ is a variable of the same size as $B$, by $\eta(\phi(x_0,x_1,..;y_0,y_1,..))=\lambda(\phi(x_0,x_1,..;B_0,B_1,..))$. By construction, $\eta$ is a measure of an indiscernible sequence. Lemma \ref{lem_meas2} yields that for any increasing $\sigma: \omega \impl \omega$, and any $\phi(x_0,x_1,..;y_0,y_1,..)$, $$\eta(\phi(x_0,x_1,..;y_{\sigma 0},y_{\sigma 1},..))=\eta(\phi(x_0,x_1,..;y_0,y_1,..)).$$ Therefore $\mu^0|_{Na}=\mu^1|_{Na}=\mu|_{Na}$. Thus $\tp(a/N)$ and $\mu|_N$ are weakly orthogonal. This proves that $\mu$ is smooth.
\end{proof}

\subsection{Reduction to dimension 1}\label{constrA}

The goal of this section is to prove the following theorem.

\begin{thm}\label{dim1}
If all sequences of elements of the home sort are distal, then $T$ is distal.
\end{thm}

We first give an informal (and incomplete) proof using measures. Assume all sequences of elements are distal and consider a generically stable measure $\mu$. Then looking at the proof of Proposition \ref{gentrans} we see that $\mu$ is weakly orthogonal to all 1-types. Then by induction, adding the points one-by-one, $\mu$ is weakly orthogonal to every $n$-type. One could make this proof rigorous, but it seems to require the fact that no type forks over its base. To avoid this hypothesis and the use of measures, we give a purely combinatorial proof.

\vspace{5pt}
\noindent
So we start with a witness of non-distality of the following form:

\begin{itemize}
\item a base set of parameters $A$, and it what follows we work over $A$ (even when not explicitly mentioned);
\item an indiscernible sequence $I=( a_i)_{i\in \mathcal I}$ with $\mathcal I=(0,1)$ (the usual interval of $\mathbb R$) for simplicity;
\item a tuple $b=(b_j)_{j<n}$, some $l\in (0,1)$ and tuple $a$ such that:\\
-- $a$ fills the cut ``$l^+$": $((a_i:i\leq l),(a_i:i>l))$ of $I$,\\
-- $I$ is $b$-indiscernible,\\
-- $I$ with $a_l$ replaced by $a$ is not indiscernible over $b$.
\end{itemize}

We make some simplifications. First let $m<n$ be the first integer such that $b'=b_{< m}$ satisfies the requirements in place of $b$. We can add $b_{<m-1}$ as parameters to the base (by base change, or equivalently we can replace $a_i$ by $a_i' = a_i\hat{~}b_{<m-1}$) and replace $b$ by $b_{m-1}$. Therefore, we may assume that $|b|=1$. Next, adding again some parameters to the base, we may assume that for $i\in \mathcal I$, $\tp(a/b)\neq \tp(a_i/b)$.

The goal of the construction that follows is to reverse the situation of $a$ and $b$, {\it i.e.}, to construct an indiscernible sequence starting with $b$ that is not distal, the non-distality being witnessed by $a$ (or a conjugate of it).

\vspace{5pt}
\noindent
\un{Step 1}: Derived sequence

\vspace{4pt}
\noindent
Let $r=\tp(a,b)$. We construct a new sequence $( a_i')_{ i\in \mathcal I}$ such that:
\begin{itemize}
\item $a_i'$ fills the cut $i^+$ of $I$;
\item $tp(a_i',b)=r$ for each $i$;
\item The sequence $\llg (a_i,a_i') : i\in \mathcal I \rrg$ is $b$-indiscernible.
\end{itemize}
This is possible by indiscernability of $(a_i)_{i\in \mathcal I}$ over $b$ (by sliding, we may choose the $a_i'$s filling the cuts and then extract).

\vspace{5pt}
\noindent
\un{Step 2}: Constructing an array

\vspace{4pt}
\noindent
Using Lemma \ref{limittype} we can iterate this construction to obtain an array $\langle a_i^n:i\in \mathcal I, n<\omega\rangle$ and sequence $\langle b_n:n<\omega \rangle$ such that:
\begin{itemize}
\item $a_i^0=a_i$ for each $i$;
\item for each $i\in \mathcal I$, $0<n<\omega$, the tuple $a_i^n$ realizes the limit type of the cut $i^+$ of $I$ over $\langle b_k,a_i^k: i\in \mathcal I, k<n\rangle$;
\item for each $0<n<\omega$, $tp(b_n, ( a_i^n)_{i\in \mathcal I }/I)=tp(b, ( a_i')_{i\in \mathcal  I } /I)$.
\end{itemize}
\noindent
\underline{Claim}: For every $\eta: \mathcal I_0\subset \mathcal I \rightarrow \omega$ injective, the sequence $\langle a_i^{\eta(i)}:i\in \mathcal I_0 \rangle$ is indiscernible, of same EM-type as $I$.
\begin{proof}
Easy, by construction.
\end{proof}

Expanding and extracting, we may assume that the sequence of rows $\langle b_n + ( a_i^n)_{i\in \mathcal I }: 0<n<\omega \rangle$ is indiscernible and that $\langle (a_i^n)_{0<n<\omega} :i\in \mathcal I \rangle$ is indiscernible over the sequence $(b_n)_{n<\omega}$.

\vspace{5pt}
\noindent
\un{Step 3}: Conclusion

\vspace{4pt}
\noindent
\un{Claim}: The sequences $(b_n)_{n<\omega}$ and $\langle (a_i^n)_{i \in \mathcal I}: 0<n<\omega \rangle$ are weakly linked (Definition \ref{weaklylinked}).
\begin{proof}
Assume for example that some $\phi(b_n,a_i^k)$ holds for all $i \in \mathcal I$ and any $(k,n)$ such that $k<n$. Take $n$ very large and take $\eta$ as in the first claim such that the truth value of ``$\eta(i) <n$" alternates more times than the alternation number of $\phi$. Then we see that $\phi(b_n,a_i^k)$ must hold also for $k>n$ (otherwise $\phi(b_n,y)$ would alternate too much on the sequence $(a_i^{\eta(i)})$). We can do something similar if the formula $\phi$ has extra parameters from the $b_n$'s or $a_i^n$'s, thus it follows that the sequences are weakly linked.
\end{proof}

\noindent
Choose an increasing map $\eta: \omega \impl \mathcal I$, then the sequences $(b_n)_{n<\omega}$ and $(a_{\eta(n)}^n)_{n<\omega}$ are weakly linked but not mutually indiscernible. This contradicts Lemma \ref{declemma} and finishes the proof of Theorem \ref{dim1}.

\begin{cor}
If all generically stable measures in dimension 1 are smooth, then all generically stable measures are smooth.
\end{cor}

This generalizes results of \cite{Finding} where this was proved under additional assumptions.

\begin{cor}\label{omin}
If $T$ is dp-minimal and has no generically stable type (in $M$), then it is distal. In particular o-minimal theories and the p-adics are distal.
\end{cor}
\begin{proof}
Recall from \ref{dpminlemme} that in a dp-minimal theory, any indiscernible sequence of elements is either distal or totally indiscernible.
\end{proof}


%
%
%
%

\subsection*{Appendix: strong honest definitions}

In a later work \cite{ExtDef2} with Artem Chernikov, we give yet another characterization of distal theories, which is probably the easiest one to use. In particular, one can obtain with it a much shorter proof of the fact that generically stable measures are smooth. We give only the statement here and refer the reader to \cite{ExtDef2} for more details.

\begin{thm}
\label{thm_stronghonest} A theory $T$ is distal if and only if the following holds:

For any $\phi(x,y)$ there is $\theta(x,z)$ such that: for any finite set $C$ and 
tuple $a$, there is $b\in C$ such that $\models\theta(a,b)$
and $\theta(x,b)\vdash\tp_{\phi}(a/C)$.
\end{thm}

\section{Domination in non-distal theories}

We have now two extreme notions for indiscernible sequences: distality and total indiscernibility. We want to understand the intermediate case. In particular, we want to show that non-distality is witnessed by stable-like phenomena. This part is essentially independent of the previous one but is of course motivated by it. We first concentrate on indiscernible sequences, and then adapt the results to invariant types. A last subsection gives an application to externally definable sets.

The reader might find it useful to have in mind the example of a colored order as defined in the introduction while reading this section.
\\

We will sometimes work with saturated indiscernible sequences, as defined below.

\begin{defi}[Saturated sequence]
An indiscernible sequence of $\alpha$-tuples is saturated if it is indexed by an $(|T|+|\alpha|)^+$-saturated dense linear order without end points.
\end{defi}

In this section, all cuts are implicitly assumed to be Dedekind ({\it i.e.}, of infinite cofinality from both sides).

If $\bar a$ fills a cut $\mathfrak c$ of $I$, an extension $J \supseteq I$ is \emph{compatible with} $\bar a$ if $\bar a$ also fills a cut of $J$.

We fix a global $A$-invariant type $p\in S_{\alpha}(\monster)$, for some small parameter set $A$. The indiscernible sequences we will consider will be Morley sequences of $p$. This is not a real restriction since every indiscernible sequence is a Morley sequence of some invariant type.

The following is the main definition of this section.

\begin{defi}[Domination]\label{stablebase}
Let $I$ be a dense indiscernible Morley sequence of $p$ over $A$, $a\models p|_{AI}$ and $\mathfrak c$ a cut of $I$ filled by a dense sequence $\bar a_* = \llg a_t:t\in \mathcal I\rrg$ of $\alpha$-tuples. We say that $\bar a_*$ dominates $a$ over $(I,A)$ if: For every cut $\mathfrak d$ of $I$ distinct from $\mathfrak c$, and $\bar b$ a dense sequence filling $\mathfrak d$, we have in the sense of $T(A)$: $$\bar b \downfree_I \bar a_*\Rightarrow \bar b \downfree_I a.$$

We say that $\bar a_*$ \emph{strongly dominates} $a$ over $(I,A)$ if for every $I\subseteq J$ compatible with $\bar a_*$ over $A$ and such that $a\models p|_{AJ}$, $\bar a_*$ dominates $\bar a$ over $J$. 
\end{defi}

We use the notation $\bar b \downfree_I a$ introduced after Definition \ref{defi_Iind} which, in this situation, means $a\models p|_{I\bar b}$.

\begin{ex}
Let $T$ be the theory of colored orders, as defined in the introduction. Let $p$ be an $A$-invariant type of an element of a new color. Let $I+a$ be a Morley sequence of $p$ over $A$. Let $\mathfrak c$ be a cut in $I$. If $a_*$ fills $\mathfrak c$, then $a_*$ dominates $a$ over $(I,A)$ if and only if $a$ and $a_*$ have the same color.
\end{ex}

\begin{lemme}\label{domsim}
The fact that $\bar a_*$ strongly dominates $a$ over $(I,A)$ only depends on the similarity class of $\tp(a,\bar a_*/I)$ over $A$.
\end{lemme}
\begin{proof}
The statement means that if $J$ is a dense indiscernible sequence, $\bar b_*$ and $b$ are tuples such that $\tp(b,\bar b_*/J)$ is similar to $\tp(a,\bar a_*/I)$ over $A$, then $\bar b_*$ strongly dominates $b$ over $(J,A)$ if and only if $\bar a_*$ strongly dominates $a$ over $(I,A)$. Take such $\bar b_*$, $b$ and $J$. Assume that $\tp(\bar b_*,b/J)$ is similar to $\tp(\bar a_*,a/I)$ over $A$. In particular, $J$ and $I$ have same EM-type over $A$, so $J$ is also a Morley sequence of $p$ over $A$. It also follows that $b\models p|_{JA}$ so its makes sense to ask for domination.

Assume that $\bar b_*$ does not strongly dominate $b$ over $(J,A)$. Then we can find a dense sequence $J'\supseteq J$ compatible with $\bar b_*$ such that $b\models p|_{J'A}$, some cut $\mathfrak d$ of $J'$ and sequence $\bar b'$ filling $\mathfrak d$ such that $\bar b' \downfree_{J'} \bar b_*$, but $\bar b' \ndownfree_{J'} b$ (all over $A$). By Corollary \ref{sliding3} (sliding), we may find $I' \supseteq I$ and $\bar a'$ such that $\tp(\bar b',\bar b_*,b/J')$ is similar to $\tp(\bar a',\bar a_*,a/I')$ over $A$. This implies the following facts:\\
-- $I'$ is compatible with $\bar a_*$ and $a\models p|_{I'A}$;\\
-- $\bar a'$ fills a cut of $I'$ distant from the cut of $\bar a_*$;\\
-- $\bar a' \downfree_{I'} \bar a_*$ and $\bar a' \ndownfree_{I'} a$.\\
Therefore $\bar a_*$ does not strongly dominate $a$ over $(I,A)$.
\end{proof}

\begin{lemme}\label{shrinkdom}
If $\bar a_*$ strongly dominates $a$ over $(I,A)$, then there is a subsequence $I'\subseteq I$ of size at most $|T|+|\alpha|$ such that $\bar a_*$ strongly dominates $a$ over $(I',A)$.
\end{lemme}
\begin{proof}
This follows from the previous lemma and Lemma \ref{shrinking} (shrinking).
\end{proof}

\begin{prop}\label{existssbase}
Let $I$ be a dense Morley sequence of $p$ over $A$ and $a\models p|_{AI}$, $\mathfrak c$ a cut of $I$ then there is a sequence of $\alpha$-tuples $\bar a_*$ of length at most $|T|+|\alpha|$ such that $\bar a_*$ fills $\mathfrak c$ and $\bar a_*$ strongly dominates $a$ over $(I,A)$.
\end{prop}
\begin{proof}
Recall the notation $\textsf T_I(a,\phi)$ from Section \ref{sec_shrinking}. If $J\subseteq J'$ are two sequences, indiscernible over $A$, then for any formula $\phi$ for which this is well defined, we have: $\textsf T_J(a,\phi) \leq \textsf T_{J'}(a,\phi)$. We will write $J \lhd J'$ if for some $\phi$, this inequality is strict.

Let $\mathfrak I$ be the class of indiscernible sequences $J$ such that one can find dense sequences $J_1$ and $J_2$ satisfying:\\
-- $J_1+J+J_2$ is a Morley sequence of $p$ over $A$;\\
-- $a \models p|_{AJ_1J_2}$.

If we have a family $(I_i)_{i<\lambda}$ of indiscernible sequences such that $I_i \subseteq I_{j}$ and $I_i \lhd I_j$ hold for all $i<j$, then taking $I_{\lambda}$ to be $\bigcup_{i<\lambda} I_i$, we have $I_i \lhd I_{\lambda}$ for all $i$. Notice in addition that if each $I_i$ belongs to $\mathfrak I$, then it is also the case for $I_{\lambda}$ (we can find $J_1$ and $J_2$ by compactness). As the numbers $\textsf T_{J'}(a,\phi)$ are finite, it follows that we can find some sequence $J$ in the class $\mathfrak I$ such that there is no $J' \supset J$ in this class with $J \lhd J'$. By shrinking, we may assume that $J$ is of size $|T|+|\alpha|$. Take $J_1$ and $J_2$ as in the definition of $\mathfrak I$. Write $\mathfrak c=(I_1,I_2)$. Without loss, $J_1$ and $J_2$ have same order types as $I_1$ and $I_2$ respectively. Composing by an automorphism over $Aa$, we may assume that $J_1=I_1$ and $J_2=I_2$. Then $J$ fits in the cut $\mathfrak c$. Set $\bar a_*=J$. 

Assume that $\bar a_*$ does not strongly dominate $a$ over $(I,A)$. Then there is a dense sequence $I'\supseteq I$ a cut $\mathfrak d$ of $I'$ and a sequence $\bar b$ filling $\mathfrak d$ such that:\\
-- $\bar a_*$ fills a cut $\mathfrak c'$ of $I'$ (over $A$);\\
-- $a\models p|_{AI'}$;\\
-- $\bar b \downfree_{I'} \bar a_*$, and $\bar b \ndownfree_{I'} a$.\\
The sequence $K=I'\cup \bar a_*\cup \bar b$ (where $\bar a_*$ and $\bar b$ are placed in their respective cuts) belongs to $\mathfrak I$. Also $\bar b \ndownfree_{I'} a$ implies that $\bar a_* \lhd K$. This contradicts maximality of $\bar a_*$ and proves that $\bar a_*$ strongly dominates $a$ over $(I,A)$.
\end{proof}

\subsubsection{External characterization and base change}

Similarly to what we did in the distal case, we give an external characterization of domination. 

\begin{prop}[External characterization of domination]\label{extstablebase}
Let $I$ be a dense Morley sequence of $p$ over $A$, $a\models p_{AI}$. Let $\bar a_*$ fill a cut $\mathfrak c$ of $I$ over $A$ such that $\bar a_*$ strongly dominates $a$ over $(I,A)$. Let also $d\in \monster$. Assume: 
\begin{description}
\item[$\boxdot$ ] There is a partition $I=J_1+J_2+J_3+J_4$ such that $J_2$ and $J_4$ are infinite, $\mathfrak c$ in interior to $J_2$, $J_2\cup \{\bar a_*\}$ is indiscernible over $Ad+ J_1+J_3+J_4$ and $J_4$ is a Morley sequence of $p$ over $Ad + J_1+J_2+J_3$.
\end{description}
Then $a \models p|_{AId}$.
\end{prop}
\begin{proof}
Let $I$, $a$, $\bar a_*$, $d$, $J_1,...,J_4$ as in the statement of the proposition. We may freely enlarge the sequence $J_2$, so we may assume that it is saturated (for example, add realizations of limit types of cuts in $J_2$ over everything. This maintains the hypothesis).

Assume $a$ does not realize $p$ over $AId$. Then there is some finite $\bar i\subset I$ and a formula $\phi(y,\bar z;x)\in L(A)$  such that $\models \phi(d,\bar i;a)$, but $p\nvdash \phi(d,\bar i;x)$. Incorporating $\bar i$ in $d$ and changing the partition so that $J_2 \cup J_4$ contains no point from $\bar i$, we may assume that $\bar i=\emptyset$. Pick a sequence of cuts of $J_2$ $\mathfrak c_0 <\mathfrak c_1 <\ldots$. Let $\llg \bar a_*^k:k<\omega\rrg$ fill the polycut $\llg \mathfrak c_k:k<\omega\rrg$ over $Ad\cup \{J_l:l\neq 2\}$, where each $\bar a_*^k$ is a sequence of same order type as $\bar a_*$. Let $I'$ denote the sequence $I$ with the points $\bar a_*^k$, $k>0$, placed in their respective cuts.

Then $\tp(\bar a_*^0,d/I')$ is similar to $\tp(\bar a_*,d/I)$. By sliding (Corollary \ref{sliding2}; note that our sequence is already large enough, so we do not need to increase it), we find $a_0$ such that: $a_0\models p|AI'$, $\phi(d;a_0)$ holds and $\bar a_*^0$ strongly dominates $a_0$ over $(I',A)$.






Let $K_1$ realize an infinite Morley sequence of $p$ over everything considered so far. Let $I_1=I \cup\{ \bar a_*^k : k>1\}+K_1$ (where the tuples $\bar a_*^k$ are placed in their respective cuts). As above, we may find $a_1 \models p|AI_1$ such that $\bar a_*^1$ strongly dominates $a_1$ over $(I_1,A)$ and $\phi(d;a_1)$ holds. Now as $a_0\downfree_{I_1} \bar a_*^1$, by the domination assumption we have $a_0 \downfree_{I_1} a_1$. We iterate this construction building an indiscernible sequence $I_\omega=I+K_1+K_2+....$ and points $\llg a_k: k<\omega\rrg$ filling the cuts between the $K_i$'s and independent over $I_\omega$ such that $\phi(d;a_k)$ holds for each $k$. As by assumption $\neg\phi(d;x)$ holds for every $x\in I_\omega$, $\phi$ has infinite alternation rank, contradicting $NIP$.
\end{proof}

\begin{prop}[Base change]\label{basechangedom}
Let $p$ be $A$ invariant and $A\subset B$. If $I$ is a dense Morley sequence of $p$ over $B$, $a\models p|BI$ and $\bar a_*$ fills a cut of $I$ in the sense of $T(B)$, then if $\bar a_*$ strongly dominates $a$ over $(I,A)$ it does so over $(I,B)$.
\end{prop}
\begin{proof}
Assume that $\bar a_*$ fills a cut $\mathfrak c$ of $I$ in the sense of $T(B)$ and dominates $a$ over $(I,A)$. Then let $\bar d$ fill a cut $\mathfrak c'$ of $I$ over $B$ with $\mathfrak c'$ distinct from $\mathfrak c$. Assume that $\bar d \downfree_I \bar a_*$ over $B$. Then $\boxdot$ holds with $d$ there replaced by $\bar dB$. By domination over $(I,A)$ and the previous proposition, $a \models p| I \cup \bar dB$. This proves that $\bar a_*$ dominates $a$ over $(I,B)$. This remains true if we first increase $I$ so $\bar a_*$ strongly dominates $a$ over $(I,B)$.
\end{proof}

\subsection{Domination for types}

We now have all we need to state domination results for types over $|T|^+$-saturated models, instead of cuts in indiscernible sequences.

We work over a fixed $\kappa$-saturated model $M$. By an \emph{invariant type} we mean here a type over $M$, invariant over some $A\subset M$ of size less than $\kappa$.

For the following definition, recall the construction of $p_x \otimes q_y$ when $q$ is invariant (Lemma \ref{stationnaire} and the paragraph following it).

\begin{defi}[Distant]
Let $p,q\in S(M)$ be two types, assume that at least one of them is invariant, then we say that $p$ and $q$ are distant if they commute: $p_x\otimes q_y=q_y\otimes p_x$\footnote{Recall the definition of commuting for non-invariant given after Lemma \ref{stationnaire}}. If $a,b\in \monster$, we wil say that $a$ and $b$ are distant over $M$ if $\tp(a/M)$ and $\tp(b/M)$ are.
\end{defi}

Keep in mind that the notion ``$a$ and $b$ are distant over $M$" only depends on $\tp(a/M) \cup \tp(b/M)$ and does not say anything more about $\tp(a,b/M)$. In particular, in a stable theory, any $a$ is distant from itself. So distant should not be confused with independent as defined now.

\begin{defi}[Independent]
Given two distant types $p,q\in S(M)$ and $a\models p$, $b\models q$ we say that $a$ and $b$ are independent over $M$ if $\tp(a,b/M)=p\otimes q$. We write $a \downfree_{M}  b$. This is a symmetric relation.
\end{defi}

\begin{defi}[S-domination]
Let $p\in S(M)$ be any type, $a\models p$. A tuple $b$ s-dominates $a$ over $M$ if:
\begin{description}
\item[$\boxminus$ ] For every invariant type $r\in S(M)$ distant from $p$ and $\tp(b/M)$, and $d\models r$, if $d \downfree_{M} b$, then $d \downfree_{M} a$.
\end{description}

\end{defi}

The reader might be concerned by the fact that this definition depends on the choice of $\kappa$ (taking a smaller $\kappa$ we have less invariant types to check). However, we will see later that we get an equivalent definition if we add in $\boxminus$ the condition that $r$ is invariant over a subset of size $\aleph_0$.

\begin{ex}
Taking again the example of a colored order, if $p$ and $q$ are two invariant types (of tuples), $\bar a\models p$ and $\bar b\models q$, then $\bar b$ s-dominates $\bar a$ over $M$ if and only if, for every point $a_0$ in range$(\bar a)$, there is a point $b_0$ in range$(\bar b)\cup M$ of the same color.
\end{ex}

\subsubsection{The moving-away lemma}

\begin{lemme}\label{existsdom}
Let $p\in S(M)$ be any type, and $a\models p$. Then there is some $a_*$ s-dominating $a$ over $M$ and furthermore $a_*$ realizes some invariant type over $M$.
\end{lemme}
\begin{proof}
This is similar to Proposition \ref{existssbase}. Start with some $a_*$ realizing an invariant type. If it does not dominate $a$, there is an invariant type $r$ distant from $a_*$ and $a$ over $M$ and $b\models r|Ma_*$ such that $b \ndownfree_M a$. Replace $a_*$ by $a_*b$ and iterate. By Corollary \ref{weight2}, this construction must stop after less than $(|T| + |a|)^+$ steps.
\end{proof}

For applications we will also need to show that we can find such a dominating tuple distant from any given type.

\begin{lemme}\label{distantcut}
Let $I\subset M$ be a dense indiscernible sequence of $\alpha$-tuples and $(I_i)_{i<\lambda}$ a family of distinct initial segments of $I$, with $\lambda\geq (|T|+|\alpha|)^+$. For $i<\alpha$, let $p_i=\lim(I_i/M)$. Then given a type $q\in S(M)$, there is $i<\lambda$ such that $p_i$ is distant from $q$.
\end{lemme}
\begin{proof}
Observe that the types $p_i$ pairwise commute. Then use Corollary \ref{weight2} (and the remark after it).
\end{proof}

\begin{lemme}\label{movingaway}
Let $p,q\in S(M)$, be types of $\alpha$-tuples ($|\alpha|<\kappa$) with $p$ invariant over some small $A$. Let $a\models p$. Then there is $r\in S(M)$ invariant over some $B$ of size $\aleph_0$, distant from $p$ and $q$ and $\bar b\models r$ such that $|\bar b|\leq |T|+|\alpha|$ and $\bar b$ s-dominates $a$ over $M$.
\end{lemme}
\begin{proof}
By Proposition \ref{existssbase} (and Lemma \ref{shrinkdom}) we can find $I'_0$ a dense Morley sequence of $p$ over $A$ of size $|T|+|\alpha|$ and $\bar a'_*$ such that $a\models p|AI'_0$, $\bar a'_*$ fills a cut $\mathfrak c$ of $I'_0$ and $\bar a'_*$ strongly dominates $a$ over $(I'_0,A)$. Let $\bar b'$ be the sequence $I'_0\cup \bar a'_*$ where $\bar a'_*$ is placed in its cut.

Let $I\subset M$ be a saturated Morley sequence of $p$ over $A$, let $\mathfrak c$ be a polarized cut of $I$ of cofinality $\aleph_0$ such that $\lim(\mathfrak c)$ is distant from $q$ and $p$ (using Lemma \ref{distantcut}). We may find some $\bar b \equiv_{Aa} \bar b'$ such that $\bar b$ fills the cut $\mathfrak c$ of $I$. Let also $I_0,\bar a_*$ be such that $(\bar b,I_0,\bar a_*)\equiv (\bar b', I'_0,\bar a'_*)$. So $\bar b=I_0 \cup \bar a_*$.

Let $I_{\infty}$ realize an infinite Morley sequence of $p$ over everything. The strong base change lemma (\ref{limittype}) works equally well if instead of considering points $d_i$ filling the cuts $\mathfrak c_i$, we take sequences $\bar d_i$. We apply this modified version with $M$ as set of parameters, $I+I_{\infty}$ as indiscernible sequence, $\bar d_0=\bar b$ and $\bar d_1=a$. We conclude that we may assume that $\bar b$ is a Morley sequence of $\lim(\mathfrak c)$ over $M$.

Set $r=\tp(\bar b/M)$ and let $B\subset M$ be of size $\aleph_0$ such that $r$ is $B$-invariant. Note that $r$ is a power of $\lim(\mathfrak c)$, so it also commutes with $p$ and $q$.

Let $d$ realize any invariant type $s\in S(M)$ distant from $p$ and $r$. Assume that $d\downfree_{M} \bar b$. Let $C\subset M$ be a subset of size $<\kappa$ such that $p,s$ and $r$ are invariant over $C$. Let $I'\subset M$ be a Morley sequence of $p$ over $C$ indexed by some dense order $\mathcal I$. Then $d\hat{~}\bar b$ realizes $s\otimes r$ over $CI'$ (indeed over $M$). As $p$ is distant from both $r$ and $s$, by associativity of $\otimes$, $p^{(\mathcal I)}$ commutes with $s\otimes r$. Therefore, $I'$ realizes $p^{(\mathcal I)}$ over $Cd\bar b$. Similarly, $\bar b$ realizes $r$ over $CI'd$, and in particular, $\bar b$ is indiscernible over $CI'd$.

Furthermore, as $I'\subset M$, $\bar b$ realizes $r$ over $CI'$. As $r$ commutes with $p$, $I'$ realizes $p^{(\mathcal I)}$ over $C\bar b$, a fortiori over $A\bar b$. But $\bar b$ is a Morley sequence of $p$ over $A$. Therefore $\bar b+I'$ is a Morley sequence of $p$ over $A$.

The hypothesis of Proposition \ref{extstablebase} are satisfied with $J_1=J_3=\emptyset$, $J_2=I_0$, $J_4=I'$ and $d$ there equal to $Cd$. We conclude that $a\models p|Cd$. As this is true for every small $C$, $d$ and $a$ are independent over $M$. This proves that $\bar b$ s-dominates $a$ over $M$.
\end{proof}

\begin{rem} The tuple $\bar b$ constructed in the previous lemma has the following additional property:

$(D)$ For every $d\in \monster$ such that $\tp(d/M\bar b)$ does not fork over $M$, and such that $\tp(\bar bd/M)$ commutes with $p$, we have $a \downfree_M d$.

This assumption is satisfied in particular when $d$ is distant from $a$ and $\bar b$, and $\bar b \downfree_M d$ (although $d$ might not realize an invariant type).

\end{rem}
\begin{proof}
We indicate how to modify the proof above. First, we take $C$ such that $p$ and $r$ are invariant over $C$. Next take $C_1$, $C\subseteq C_1 \subset M$, such that for any $J,J' \subset M$ Morley sequences of $p$ over $C_1$ indexed by $\omega$, we have $\tp(J/C\bar b d)=\tp(J'/C\bar bd)$. This is possible using Lemma \ref{stationnaire}. Build $I'$ as a Morley sequence of $p$ over $C_1$. By definition of commuting, $I'$ is a Morley sequence of $p$ over $C\bar bd$. Also because $\tp(d/M\bar b)$ does not fork over $M$, $\bar b$ is indiscernible over $Md$. Finally, the proof that $\bar b+I'$ is a Morley sequence of $p$ over $A$ does not change. So as above, we may apply Proposition \ref{extstablebase} to conclude that $d$ and $a$ are independent over $M$.
\end{proof}

\begin{cor}\label{movingaway2}
Let $p,q\in S(M)$ be any two types of $\alpha$-tuples ($|\alpha|<\kappa$) and let $a\models p$. Then there is $a_*$ a tuple of length $\leq |T|+|\alpha|$, distant from $q$ over $M$ and such that $a_*$ s-dominates $a$ over $M$. Furthermore, we may assume that $\tp(a_*/M)$ is invariant over a subset of size $\aleph_0$.
\end{cor}
\begin{proof}
By Lemma \ref{existsdom}, there is some $a_{**}$ s-dominating $a$ over $M$ and realizing some invariant type. By Lemma \ref{movingaway}, there is a tuple $a_*$ s-dominating $a_{**}$ over $M$ with the required size, whose type over $M$ is invariant over a subset of size $\aleph_0$ and distant from $q$.

We check that $a_{*}$ s-dominates $a$ over $M$. Let $r\in S(M)$ be an invariant type distant from $a_{*}$ and $a$. Let $b\models r$ with $b \downfree_M a_{*}$. By Lemma \ref{movingaway}, there is $b_*$ s-dominating $b$ and distant from $q=\tp(a\hat{~}a_{*}\hat{~}a_{**}/M)$. Furthermore assume that $b_*$ satisfies property $(D)$. Composing by an automorphism over $Mb$, we may further assume that $b_* \downfree_M a_{*}$. Then as $a_{*}$ s-dominates $a_{**}$ over $M$, we have $b_* \downfree_M a_{**}$ and as $a_{**}$ s-dominates $a$ over $M$, $b_*\downfree_M a$. By property $(D)$ this implies $b\downfree_M a$.
\end{proof}

\begin{lemme}[Transitivity of s-domination]\label{transitivity}
Let $a\in \monster$ and let $a_*$ s-dominate $a$ over $M$. Let also $a_{**}$ s-dominate $a_*$ over $M$. Then $a_{**}$ s-dominates $a$ over $M$.
\end{lemme}
\begin{proof}
Let $d\in \monster$ be distant from $a$ and $a_{**}$ with $d\downfree_M a_{**}$. By Corollary \ref{movingaway2}, let $d_*$ s-dominate $d$ over $M$ and distant from $a\hat{~}a_*\hat{~}a_{**}$. Composing by an automorphism over $Md$, we may assume that $d_* \downfree_M a_{**}$. Then we have $d_* \downfree_M a_*$ and $d_* \downfree_M a$ and finally $d \downfree_M a$.
\end{proof}

\begin{ex}\label{exstabledistal}
If $p\in S(M)$ is generically stable, and $a\models p$, then $a$ is s-dominated by itself. In the opposite situation, if $p$ is invariant and its Morley sequence is distal, then $a$ is s-dominated by the empty set.
\end{ex}

\subsubsection{S-independence}

\begin{defi}[S-independence]
Let $p,q$ be any types over $M$, let $a\models p$ and $b\models q$. We say that $a$ and $b$ are \emph{s-independent} over $M$ and write $a \downfree^s_{M} b$ if there is a tuple $a_*$ realizing an invariant type, s-dominating $a$ and distant from $b$ such that $a_*\downfree_{M} b$.
\end{defi}

Note that if $a$ and $b$ are distant, then $a \downfree^s_{M} b$ if and only if $a \downfree_{M} b$.

\begin{prop}[Existence]
Let $p,q \in S(M)$ be any two types and $a\models p$. Then there is $b\models q$ such that $a \downfree_M^s b$.
\end{prop}
\begin{proof}
Let $a_*$ be s-dominating $a$ such that $a_*$ realizes some invariant type $p_*$ distant from $q$. Take $b$ such that $\tp(a_*,b/M)=p_* \otimes q$. Then by definition $a \downfree^s_M b$.
\end{proof}

\begin{prop}[Symmetry of s-independence]
S-independence is symmetric: if $a$ and $b$ are two tuples, then $a \downfree_{M}^s b$ if and only if $b \downfree_{M}^s a$ if and only if there are $a_*$, $b_*$ s-dominating $a$ and $b$ respectively, distant from each other such that $a_* \downfree_{M} b_*$.
\end{prop}
\begin{proof}
It is enough to prove the last equivalence. To see right to left, let $a_{**}$ s-dominate $a_*$ and be distant from $b_*$ and $b$ over $M$. Assume also that $a_{**} \downfree_M b_*$, then by Lemma \ref{transitivity}, $a_{**}$ s-dominates $a$ over $M$. As it is independent from $b_*$ over $M$, we have $a_{**} \downfree_M b$ as required.

Conversely, assume that $a\downfree_{M}^s b$. Let $a_*$ be a tuple s-dominating $a$, realizing an invariant type over $M$, and distant from $b$ such that $b \downfree_{M} a_*$. We can find a tuple $b'_*$ s-dominating $b$ distant from $a,a_*$ and $b$. As $a_*\downfree_{M} b$, there is $b_* \equiv_{Mb} b'_*$ such that $a_* \downfree_{M} b_*$.
\end{proof}

\begin{prop}[Weight is bounded]\label{bddweight}
Let $(b_i)_{i<|T|^+}$ be a sequence of tuples such that $b_i \downfree_M^s b_{<i}$ for each $i$, and let $a\in \monster$. Then there is $i<|T|^+$ such that $a \downfree^s_{M} b_i$.
\end{prop}
\begin{proof}
By Lemma \ref{movingaway2}, we can find a family $(b^*_i)_{i<|T|^+}$ such that: For each $i<|T|^+$, $b^*_i$ realizes an invariant type $r_i$ distant from $q:=\tp(a/M)$ and $r_j$, $j\neq i$, $b^*_i$ s-dominates $b_i$ over $M$ and $b^*_i\downfree_M b^*_{<i}$. By Corollary \ref{weight2}, there is $i<|T|^+$ such that $\tp(b^*_i,a/M)=r_i \otimes q$. By definition, $a \downfree^s_{M} b_i$.
\end{proof}

The following special case of this proposition makes no reference to s-domination.

\begin{cor}\label{bddweight2}
Let $q\in S(M)$ be $A$-invariant and, for $i<|T|^+$, let $p_i\in S(M)$ be an invariant type. Assume that $p_i$ commutes with $q$, for each $i$. Let $(b_i) \models \bigotimes p_i$ and $a \models q$. Then there is $i<|T|^+$ such that $\tp(b_i,a/N)=p_i \otimes q$.
\end{cor}

\begin{cor}
Let $a,b\in \monster$ such that $a \ndownfree^s_{M} b$, then $\tp(b/Ma)$ forks over $M$.
\end{cor}
\begin{proof}
Otherwise, we could find a global $M$-invariant extension $\tilde p$ of $\tp(b/Ma)$. Take $(a_i)_{i<|T|^+}$ to be a sequence of realizations of $\tp(a/M)$ with $a_0=a$ and $a_i \downfree^s_M a_{<i}$ for each $i$. By invariance, if $b_* \models \tilde p$ over everything, for each $i<|T|^+$, $\tp(b_*,a_i/M)=\tp(b_*,a/M)$ and $b_* \ndownfree^s_{M} a_i$. This contradicts Proposition \ref{bddweight}.
\end{proof}

\begin{cor}
Let $a$ and $b$ be distant over $M$, then $\tp(a/Mb)$ forks over $M$ if and only if $\tp(b/Ma)$ forks over $M$ if and only if $a \ndownfree_{M} b$.
\end{cor}

\begin{prop}\label{reflexion}
Let $p\in S(M)$ be an invariant type and $q\in S(M)$ be distant from $p$. Let $I=(a_i)_{i<\omega}$ be a Morley sequence of $p$ over $M$ and $b\models q$. Then $\lim(I/Mb)=p|_{Mb}$.
\end{prop}
\begin{proof}
This follows easily from  Proposition \ref{bddweight} by making the sequence $I$ of large cardinality.
\end{proof}

\begin{ex}[ACVF]\label{exacvf}
Take $T$ to be ACVF, and $M$ a model of $T$. Let $p\in S(M)$ be an invariant type of a field element. By \cite{HHM2}, Corollary 12.14, there are definable functions $f$ and $g$ respectively into the residue field $k$ and the value group $\Gamma$ such that letting $p_k = f_*(p)$ and $p_{\Gamma}= g_*(p)$, we have:

For any $a\models p$ and $b\in \monster$, $\tp(a/Mb)=p|_{Mb}$ if and only if $\tp(f(a)/Mb)=p_k|_{Mb}$ and $\tp(g(a)/Mb)=p_{\Gamma}|_{Mb}$.

Take such an invariant type $p$ and $a\models p$. Then $a$ is s-dominated by $f(a)$ since if $b\in \monster$ is distant from $a$ over $M$, then by distality of $\Gamma$, $\tp(b/M)$ and $\tp(g(a)/M)$ are weakly orthogonal.
\end{ex}

\subsection{The finite-co-finite theorem and application}

We prove now an analog of Proposition \ref{bddweight} which does not require to work over a model. We prove it by reproducing the proof of that proposition in the context of domination for indiscernible sequences.

\begin{prop}\label{fcof}
Let $A$ be any set of parameters and let $p$ be some global $A$-invariant type. Let $a\in \monster$. Let $I$ be an infinite Morley sequence of $p$ over $Aa$ and $J$ be an infinite Morley sequence of $p$ over $AI$. Let $\phi(x;y)\in L(A)$, then the set $\{b\in J : \models \phi(b,a)\}$ is finite or co-finite in $J$.
\end{prop}
\begin{proof}
Assume not. Then we may expand $I$ to a saturated sequence. Without loss, the formula $\phi(x,b)$ is true for $x\in I$ and pruning $J$, we may assume that it is false for $x\in J$. Finally, we may expand $J$ so that $J=\llg b_i : i<|T|^+ \rrg$.

We can find sequences $\llg \bar b_*^i : i <|T|^+ \rrg$ such that :
\\
-- Each $\bar b_*^i$ fills some cut of $I$, the cuts being distinct from one another, and the $\bar b_*^i$ are placed independently over $I$;
\\
-- for each index $i$, $\bar b_*^i$ strongly dominates $b_i$ over $(I,A)$.
\\
(Why ? First take $\bar d_*^0$ strongly dominating $b_0$ over $(I,A)$. Let $\llg b'_i : 0<i<|T|^+\rrg$ be a Morley sequence of $p$ over everything. There is an automorphism $\sigma$ fixing $AIb_0$ sending $\llg b'_i : 0<i<|T|^+\rrg$ to $\llg b_i : 0<i<|T|^+\rrg$. Let $\bar b_*^0 = \sigma(\bar d_*^0)$.  Then take $\bar d_*^1$ strongly dominating $b_1$ over $(I,A)$ with $\bar d_*^1 \downfree_I \bar b_*^0$. And iterate.)

Let $I'$ be the sequence $I$ with all the $\bar b_*^i$ added in their respective cuts. It is an $A$-indiscernible sequence. By shrinking of indiscernibles, there is $I'' \subseteq I$ obtained by removing at most $|T|$ of the tuples $\bar b_*^i$ from $I'$ such that $I''$ is indiscernible over $Aa$. Without loss, assume we have not removed the tuple $\bar b_*^0$. Then by Proposition \ref{extstablebase} (External characterization), $b_0\models p|_{Aa}$. This contradicts the hypothesis.
\end{proof}

\begin{thm}[Finite-co-finite theorem]\label{whatcanhappen}
Let $I=I_1+I_2+I_3$ be indiscernible, $I_1$ and $I_3$ being infinite. Assume that $I_1+I_3$ is $A$-indiscernible and take $\phi(x;a)\in L(A)$, then the set $B=\{ b \in I_2 : \models \phi(b;a)\}$ is finite or co-finite.
\end{thm}
\begin{proof}
This follows from the previous proposition by setting $p$ to be the limit type of $I_3 ^*$ ($I_3$ in reverse order).
\end{proof}

Note that necessarily, $B$ in the statement of the theorem is finite if $\neg \phi(b;a)$ holds for $b\in I_1+I_3$ and co-finite otherwise (because you can incorporate some parts of $I_1$ and $I_3$ to $I_2$, also it follows from the proof). This will be used implicitly in applications.

\begin{cor}\label{fcofcor1}
Let $I=I_1+I_2+I_3$ be indiscernible, $I_1$ and $I_3$ being infinite with no endpoints and $I_2$ densely ordered. Assume that $I_1+I_3$ is $A$-indiscernible. Write $I_2 = (a_i)_{i\in \mathcal I}$. Then given some linear order $\mathcal J \supseteq \mathcal I$, one can find tuples $a_i$, $i \in \mathcal J\setminus \mathcal I$ such that:
\\
-- $I_1 + \llg a_i : i\in \mathcal J \setminus \mathcal I\rrg + I_3$ is indiscernible over $A$,
\\
-- $I_1 + \llg a_i : i\in \mathcal J \rrg + I_3$ is indiscernible.
\end{cor}
\begin{proof}
We construct the points $a_i$, $i\in \mathcal J \setminus \mathcal I$ simply by realizing limit types of cuts of $I_2$ over everything. More precisely, given $\mathfrak c$ a cut of $\mathcal I$, identify $\mathfrak c$ with the corresponding cut of $I_2$. Assume for simplicity that $\mathfrak c$ has infinite cofinality from the right and let $p_{\mathfrak c}$ be $\lim(\mathfrak c^+)$ (seen a global type). Note that if $\mathfrak c\neq \mathfrak c'$, then the types $p_{\mathfrak c}$ and $p_{\mathfrak c'}$ commute. Let $\mathcal J_{\mathfrak c}$ be the convex subset of $\mathcal J$ formed by elements falling in the cut $\mathfrak c$. Finally take $\llg a_i : i\in \mathcal J \setminus \mathcal I \rrg$ to realize $\bigotimes_{\mathfrak c} p_{\mathfrak c}^{(\mathcal J_{\mathfrak c})}$ over $IA$.
 
The second condition is obviously satisfied, so we have to check the first one. We start by considering a cut $\mathfrak c$, and show that $I_1 + \llg a_i: i\in \mathcal J_{\mathfrak c}\rrg + I_3$ is indiscernible over $A$. The fact that for $i\in \mathcal J_{\mathfrak c}$, and $a\in I_1$, $\tp(a_i/A)=\tp(a/A)$ follows immediately from the finite-co-finite theorem \ref{whatcanhappen}. Now consider $i<j \in \mathcal J_{\mathfrak c}$ and $\phi(x_1,x_2)\in L(A)$ a formula. Assume that for $a\in I_1$, $b\in I_3$ we have $\models \phi(a,b)$. Write $c=(\mathcal K_1,\mathcal K_2)$, seen as a cut of $\mathcal I$. By construction of $(a_i)_{i\in \mathcal J_{\mathfrak c}}$ and shrinking of indiscernibles (Proposition \ref{shrinking1}), we have:
$$\models \phi(a_i,a_j) \iff \text{for some coinitial }\mathcal K\subset \mathcal K_2, \phi(a_s,a_t)\text{ holds for }s<t \in \mathcal K.$$

Assume we have $\neg \phi(a_i,a_j)$. So easily, we can find points $s_1<t_1<s_2<t_2<... \in \mathcal K_2$ such that $\neg \phi(a_{s_k},a_{t_k})$ holds for each $k<\omega$. Let $L_2=\llg a_{s_k}\hat{~} a_{t_k}:k< \omega\rrg$. Take also $L_1$ to be any sequence of increasing pairs of members of $I_1$, so that $L_1+L_2$ is indiscernible, and pick similarly $L_3$. Then the finite-co-finite theorem applied to the sequence $L_1+L_2+L_3$ gives us a contradiction.

We can do the same reasoning if $\phi(x_1,x_2)$ has parameters in $AI_1I_2$ (by adding parts of $I_1I_2$ to $A$ and decreasing them). Also one sees at once that the construction generalizes to formulas $\phi(x_1,...,x_n)$ with more variables and we obtain than $I_1+\llg a_i : i\in \mathcal J_{\mathfrak c} \rrg + I_3$ is indiscernible over $A$.

Next, we look at two cuts $\mathfrak c_1 < \mathfrak c_2$ and we want to see that $I_1+\llg a_i : i\in \mathcal J_{\mathfrak c_1}+\mathcal J_{\mathfrak c_2} \rrg + I_3$ is indiscernible over $A$. We know that $\llg a_i : i\in \mathcal J_{\mathfrak c_2} \rrg$ realizes $p_{\mathfrak c_2}^{(\mathcal J_{\mathfrak c_2})}$ over everything else. We may assume that $\mathcal J_{\mathfrak c_1}$ is without endpoints. Take some finite $\mathcal K_0\subset \mathcal J_{\mathfrak c_1}$ and let $\mathcal K_1$ be $\{i\in \mathcal J_{\mathfrak c_1} : i>\mathcal K_0\}$. Then the sequence $\llg a_i: i\in \mathcal K_1\rrg+I_3$ is indiscernible over $A\cup \{a_i : i\in \mathcal K_0\}$. The same reasoning as above shows that the sequence $\llg a_i : i\in \mathcal K_1 \rrg+\llg a_i : i\in \mathcal J_{\mathfrak c_2} \rrg +I_3$ is indiscernible over $A\cup \{a_i : i\in \mathcal K_0\}$. It follows that $I_1+\llg a_i:i\in \mathcal J_{\mathfrak c_1}+\mathcal J_{\mathfrak c_2}\rrg+I_3$ is indiscernible over $A$.

Iteratively, we prove that $I_1+\llg a_i :i \in \mathcal J_{\mathfrak c_1}+... +\mathcal J_{\mathfrak c_n} \rrg +I_3$ is indiscernible over $A$ and finally, that $I_1+\llg a_i : i \in \mathcal J\setminus \mathcal I \rrg +I_3$ is indiscernible over $A$.
\end{proof}

\begin{cor}\label{fcofcor2}
Let $I_1 + I_2+ I_3$ be an indiscernible sequence of finite tuples, with $I_1$ and $I_3$ infinite without endpoints. Assume that $I_1+I_3$ is indiscernible over $A$. Then we can find some subsequence $I_2'\subset I_2$ with $I_2 \setminus I_2'$ of size at most $|T|+|A|$ such that $I_1+I_2'+I_3$ is indiscernible over $A$.
\end{cor}
\begin{proof}
Without loss, we may assume that $I_2$ is densely ordered. Write $I_2=\llg a_i:i\in \mathcal I\rrg$ and take some $|\mathcal I|^+$-saturated linear order $\mathcal J\supset \mathcal I$. By Corollary \ref{fcofcor1} we can find tuples $\llg a_i : i\in \mathcal J \setminus \mathcal I\rrg$ such that :
\\
-- $I_1 + \llg a_i : i\in \mathcal J \setminus \mathcal I\rrg + I_3$ is indiscernible over $A$,
\\
-- $I_1 + \llg a_i : i\in \mathcal J \rrg + I_3$ is indiscernible.

By shrinking of indiscernibles, there is $\mathcal J_0\subset \mathcal J$ of size at most $|T|+|A|$ such that $I_1 + \llg a_i:i\in \mathcal J \setminus \mathcal J_0\rrg + I_3$ is indiscernible. Then set $I_2' = \llg a_i : i\in \mathcal I \setminus \mathcal J_0 \rrg$.
\end{proof}

We now give an application of this result to externally definable sets.


We will use the following notation: if $M\models T$, $M\prec N$ is an elementary extension and $A\subseteq N$ containing $M$, then $M_{[A]}$ is the structure with universe $M$ with language composed of a predicate for every subset of $M^l$ (any $l$) of the form $\phi(M;\bar c)$, $\bar c\in A^k$ for any $\phi(\bar x;\bar y)\in L(M)$, interpreted the obvious way.

Shelah proved in \cite{Sh783} that $M_{[\monster]}$ eliminates quantifiers. We refer the reader to \cite{DepPairs} for a slightly different approach, that we will use (and recall) here. If $p\in S(M)$ is any type and $a \models p$, then it is not true in general that $M_{[a]}$ eliminates quantifiers (see \cite{DepPairs}, Example 1.8 for a counterexample). However it is conjectured in \cite{DepPairs} that $M_{[I]}$ does, where $I$ is a coheir sequence starting with $a$. We prove a special case of this when $p$ is interior to $M$. See the definition below.

We will need some notions from \cite{DepPairs} that we recall now. If $X$ is an externally definable subset of $X$ ({\it i.e.}, a subset of the form $\phi(M,c)$ for some tuple $c\in \monster$), then an \emph{honest definition} of $X$ is a formula $\theta(x,d) \in L(\monster)$ such that (1) $\theta(M,d)=X$ and (2) for every formula $\psi(x)\in L(M)$ such that $X \subseteq \psi(M)$ then $\monster \models \theta(x) \impl \psi(x)$.

\begin{lemme}
If $A\subset \monster$ containing $M$ is such that for every formula $\phi(x;c)\in L(A)$, $\phi(M;c)$ has an honest definition with parameters in $A$, then $M_{[A]}$ eliminates quantifiers.
\end{lemme}
\begin{proof}
Let $\phi(x,y;c)\in L(A)$ and let $\theta(x,y;d)\in L(A)$ be an honest definition of $X:=\phi(M;c)$. Let $\pi : M^{|x|+|y|} \impl M$ be the projection on the first $|x|$ coordinates. Let $\psi(x;d) = \exists y\theta(x,y;d)$. Then $\psi(M;d)=\pi(X)$: it is clear that $\psi(M;d)\subseteq \pi(X)$, and if $a\in M^{|x|}\setminus \pi(X)$, then the set $\{(x,y)\in M^{|x|+|y|} : y \neq a\}$ contains $X$ and by honesty $\monster \models \theta(x,y;d) \impl y\neq a$ which gives the reverse inclusion.
\end{proof}

\begin{defi}
Let $p$ be an $M$-invariant global type. We say that $p$ is \emph{interior} to $M$ if $p^{(\omega)}$ is both an heir and a co-heir of its restriction to $M$.
\end{defi}

An example of an interior type is given by the following situation: let $I\subset M$ be indiscernible and $\mathfrak c$ a cut interior to $I$ such that $M$ respects $\mathfrak c$. Then the type $p=\lim(\mathfrak c^+)$ is interior to $M$.

\begin{lemme}\label{intlemma}
Let $p$ be a global $M$-invariant type interior to $M$. Let $I_0+I_1+I_2$ be a Morley sequence of $p$ over $M$. For $i<3$ let $\bar a_i \subset I_i$ be a finite tuple. Assume that $\bar a_1 \models \phi(\bar x;\bar a_0,\bar a_2)$, $\phi\in L(M)$, then there are two tuples $\bar b_0, \bar b_2 \subset M$ such that $\bar a_1 \models \phi(\bar x;\bar b_0, \bar b_2)$.
\end{lemme}
\begin{proof}
First find $\bar b_2$ such that $\bar a_1 \models \phi(\bar x;\bar a_0,\bar b_2)$ by the coheir hypothesis. Then find $\bar b_0$ by the heir hypothesis.
\end{proof}

\begin{thm}[Shelah expansion for interior types]\label{shexpansion}
Let $p$ be a global $M$-invariant type interior to $M$. Let $I$ be a Morley sequence of $p$ over $M$. Then $M_{[I]}$ eliminates quantifiers.
\end{thm}
\begin{proof}
Take a saturated extension $M_{[I]} \prec N^*$ of size $\kappa > |M|$. The model $N^*$ can be seen as a reduct to the language of $M_{[I]}$ of some $N_{[J]}$ for $M\prec N$ and $J \equiv_M I$, $J$ indiscernible over $N$. Without loss $I=J$. Notice that $N^*$ and $N_{[I]}$ have the same definable sets.

\vspace{5pt}
\noindent
\underline{Claim}: There is an indiscernible sequence $I_1+I_2 \subset N$ such that $N$ respects the cut $c=(I_1,I_2)$ and $I\models \lim(c^+)^{(\omega)}$.
\vspace{4pt}

Proof: Write $N=\bigcup_{i<\kappa} A_i$ with $|A_i|<\kappa$. Let $i<\kappa$. By Lemma \ref{intlemma} and saturation, we can find sequences $K_i, L_i \subset N$ of order type $\omega$ such that $K_i+I+L_i$ is indiscernible over $A_i$.
Let $I_1 = K_1+K_2+...$ and $I_2 = ... + L_2 +L_1$, the sums ranging over $i<\kappa$. The required property is then easy to check.

\vspace{4pt}

Let $\phi(x;y)$ be a formula and $a_0\models p$, $a_0 \in I$. We consider the pair $(M,N)$ and show that $\phi(a_0;M)$ has an honest definition with parameters in $M +I_1 + I_2$.

By the Theorem \ref{whatcanhappen} and compactness, there are integers $n,N$ and a finite set of formulas $\delta$ such that for every finite sequence $J_1+J_3+J_2$, satisfying:\\
-- $J_1$ and $J_2$ are of size at least $n$,\\
-- $J_1+J_3+J_2$ is indiscernible,\\
-- $J_1+J_2$ is $\delta$-indiscernible over $b$ and \\
-- $\phi(x;b)$ holds on all elements of $J_1$ and $J_2$,\\
then $\neg \phi(x;b)$ holds on at most $N$ elements of $J_3$.

Let $I_1' \subset I_1$ and $I_2' \subset I_2$ be finite of size $n$ such that $I_1'+I_2'$ is $M$-indiscernible. Consider the formula $\theta(y) \in L(MI)$ such that if $b \models \theta(y)$, then $I_1'+I_2'$ is $\delta$-indiscernible over $a$, and $\phi(\bar a_0;y)$ holds on all elements of $I_1'+I_2'$. Define analogously $\theta_1(y)$ using $\neg \phi$ instead of $\phi$.

Then, for every $b\in M$, $\theta(b)$ holds if and only if $\phi(a_0;b)$ holds. Also, if $b\in N$, and $\theta(b)$ holds, then $\phi(a_0;b)$ holds (Why ? Only finitely many elements $a$ from $I_1+I_2$, with $I_1'<a<I_2'$ can satisfy $\phi(a;b)$). This easily implies that $\theta$ is an honest definition of $\phi(a_0;M)$.

To conclude the theorem, notice that we can do the same thing replacing $p$ by $p^{(n)}$ for any $n$, which takes care of formulas $\phi(\bar a;y)$ with $\bar a$ a finite subset of $I$ instead of one element.
\end{proof}
%
%

\section{Sharp theories}

In this last section, we study theories in which types are s-dominated by generical\-ly stable types. We show that this is implied by the existence of some form of decomposition of indiscernible sequences into ``stable by distal". Our goal is to give a criterion which we can check in dimension 1 and conclude that dp-minimal theories are sharp. One could probably introduce stronger notions, and ask for example that types are s-dominates by types living in a stable sort, but we do not pursue this here.

\begin{defi}
The theory $T$ is sharp if for every $|T|^+$-saturated model $M$ and $p\in S(M)$ an invariant type realized by $a$, there is some generically stable type $q\in S(M)$ and $a_*\models q$ such that $a_*$ s-dominates $a$ over $M$.
\end{defi}

\begin{defi}
Let $I=\llg a_i:i\in \mathcal I\rrg$ be a dense indiscernible sequence. A decomposition of $I$ is an indiscernible sequence $K=\llg a_i\hat{~}b_i : i\in \mathcal I\rrg$ where the sequence $J=(b_i)_{i\in \mathcal I}$ is totally indiscernible and such that:

For every $K'$ of same EM-type as $K$, $\mathfrak c$ a Dedekind cut of $K'$, $d\in \monster$ such that $K'$ is indiscernible over $d$ and $a\hat{~}b$ filling $\mathfrak c$; if there is $a'$ such that $a'\hat{~}b$ fills $\mathfrak c$ over $dK'$, then $a\hat{~}b$ fills $\mathfrak c$ over $dK'$.
\end{defi}

By usual sliding argument, if $K$ is dense and contains some Dedekind cut $\mathfrak c$, it is enough to check the condition for $K'=K$.

An indiscernible sequence $I=\llg a_i:i\in \mathcal I\rrg$ is \emph{decomposable} if it admits a decomposition $K=\llg a_i\hat{~}b_i:i\in \mathcal I\rrg$. In this case, we will say that $I$ is decomposable over $\llg b_i:i\in \mathcal I\rrg$.

\begin{rem}
There are two trivial cases of decomposability: If $I$ is distal, then it is decomposable over the sequence of empty tuples, if $I$ is totally indiscernible, it is decomposable over itself.
\end{rem}

\begin{lemme}[Internal characterization]
An indiscernible sequence of pairs $I=(a_i\hat{~}b_i)_{i\in \mathcal I}$ is a decomposition, if and only if the following holds:

$\boxtimes$ For every $J, K, L$ infinite indiscernible sequences without endpoints of same EM-type as $I$ and $a\hat{~}b$, $a'\hat{~}b'$, if $J+a\hat{~}b+K+L$, $J+K+a'\hat{~}b'+L$ are indiscernible, and there exist $a_0,a_0'$ such that $J+a_0\hat{~}b+K+a'_0\hat{~}b'+L$ is indiscernible, \un{then} $J+a\hat{~}b+K+a'\hat{~}b'+L$ is indiscernible.
\end{lemme}
\begin{proof}
Assume that $I$ is a decomposition. Then taking $d=a'_0\hat{~}b'+L$ in the definition, we see that $J+a\hat{~}b+K$ is indiscernible over $a'_0\hat{~}b'+L$. Then taking $d=J+a\hat{~}b$, we get that $K+a'\hat{~}b'+L$ is indiscernible over $J+a\hat{~}b$, so $J+a\hat{~}b+K+a'\hat{~}b'+L$ is indiscernible.

Conversely, assume $\boxtimes$ holds and without loss $\mathcal I$ is a dense order. Notice that the analog of $\boxtimes$ where we fill $n$ cuts instead of 2 follows from $\boxtimes$ by easy induction (as in Lemma \ref{twoton}). Let $d \in \monster$, $\mathfrak c$, $a\hat{~}b$ and $a'$ be as in the definition of decomposition. Assume that $a\hat{~}b$ does not fill $\mathfrak c$ over $Ad$. Adding parameters to $d$ if necessary, we may assume that for some formula $\phi(x,y)$, and all $a_*\hat{~}b_*\in I$, we have $\phi(a_*\hat{~}b_*,d)\wedge \neg \phi(a\hat{~}b,d)$. Fix some increasing sequence $(\mathfrak c_k)_{k<\omega}$ of Dedekind cuts of $I$. For each $k<\omega$, we can find $a_k,a_k',b_k$ such that $\tp(a_k,a_k',b_k,d/I)$ is similar to $\tp(a,a',b,d/I)$ and $a_k \hat{~} b_k$ fills the cut $\mathfrak c_k$. By $\boxtimes$ and the remark above, the sequence obtained by adding all the tuples $a_k \hat{~} b_k$ to $I$ in their respective cuts is indiscernible. Then the formula $\phi(x,y)$ has infinite alternation rank.
\end{proof}

We will need the following strengthening of Lemma \ref{limittype}.

\begin{lemme}[Strong base change 2]\label{limittype2}
Let $I=(a_i\hat{~}b_i)_{i \in \mathcal I}$ be an indiscernible sequence and $A\supset I$ a set of parameters. Let $(\mathfrak c_i)_{i \in \mathcal J}$ be a sequence of pairwise distinct polarized Dedekind cuts in $I$. Call $\mathfrak c'_i$ the corresponding cut in the sequence $(b_i)_{i\in \mathcal I}$. For each $i$ let $d_i\hat{~}e_i$ fill the cut $\mathfrak c_i$. Assume also that the sequence $(e_i)_{i \in J}$ realizes $\bigotimes \lim(\mathfrak c'_i)$ over $I$.
Then there exist $(d_i'\hat{~}e_i')_{i \in \mathcal J}$ such that 
\\
-- $\tp(\langle d_i'\hat{~}e_i':i\in \mathcal J\rangle/I)=\tp(\langle d_i\hat{~}e_i:i\in \mathcal J\rangle/I)$,
\\
-- for each $i$, $\tp(d_i'\hat{~}e_i'/A) = \lim(\mathfrak c_i/A)$,
\\
-- $(e_i')_{i\in \mathcal J}$ realizes $\bigotimes_i \lim(\mathfrak c'_i)$ over $A$.
\end{lemme}
\begin{proof}
The proof is essentially the same as that of Lemma \ref{limittype}.

Assume the result does not hold. Then by compactness, we may assume that $\mathcal J=\{1,..,n\}$ and that there is a formula $\phi(x_1\hat{~}y_1,..,x_n\hat{~}y_n) \in \tp(\langle d_i\hat{~}e_i:i \rangle/I)$, a formula $\theta(y_1,..,y_n) \in \bigotimes \lim(t'_i/m)$ and formulas $\psi_i(x_i,y_i) \in \lim(\mathfrak c_i/m)$ for some finite $m \in A$ such that $\phi(x_1\hat{~}y_1,..,x_n\hat{~}y_n)\wedge \theta(y_1,..,y_n) \bigwedge_i \psi_i(x_i,y_i)$ is inconsistent. Let $I_0$ denote the parameters of $\phi$.

Assume for simplicity that $n=2$ (the proof for $n>2$ is the same) and without loss $\mathfrak c_i$ is polarized as $\mathfrak c_i^-$. For $i=1,2$  take $(I_i,I_i') \unlhd \mathfrak c_i$ such that $\psi_i$ holds on all elements of $I_i$, $\theta(y_1,y_2)$ holds for each $(x_1\hat{~}y_1,x_2\hat{~}y_2)\in I_1 \times I_2$, and $I_i \cup I_i'$ contains no element of $I_0$. Then $I_1+I_1'$ and $I_2+I_2'$ are mutually indiscernible over $I_0$. So for every two cuts $\mathfrak d_1$ and $\mathfrak d_2$ respectively from $I_1+I_1'$ and $I_2+I_2'$, we can find points $d_1\hat{~}e_1$ and $d_2\hat{~}e_2$ filling those cuts (even seen as cuts of $I$) such that $\phi(d_1,e_1,d_2,e_2)$ holds and there are $d_1',d_2'$ such that $(d_1'\hat{~}e_1,d_2'\hat{~}e_2)$ fills the polycut $(\mathfrak d_1,\mathfrak d_2)$ over $I$.

Take a cut $\mathfrak d_1$ inside $I_1$ and $\mathfrak d_2$ inside $I_2$ and see them as cuts of $I$. Fill $\mathfrak d_1$ by $d_1\hat{~}e_1$ and $\mathfrak d_2$ by $d_2\hat{~}e_2$ as above. By hypothesis, either $\neg \theta(e_1,e_2)$, $\neg \psi_1(d_1,e_1)$ or $\neg \psi_2(d_2,e_2)$ holds. In one of the latter two cases, proceed as in Lemma \ref{limittype}. In the first case, keep $e_1$ and $e_2$ and add points $(d_1',d_2')$ such that $I$ with $d_1'\hat{~}e_1$ and $d_2'\hat{~}e_2$ added is indiscernible. Then iterate with $I\cup \{d_1'\hat{~}e_1,d_2'\hat{~} e_2\}$ instead of $I$.

After iterating this $\omega$ times, either $\psi_1$, $\psi_2$ or $\theta$ has infinite alternation rank.
\end{proof}

\begin{lemme}[Base change]
The notion of being a decomposition is stable both ways under base change: If $(a_i\hat{~}b_i)_{i\in \mathcal I}$ is $A$-indiscernible, then it is a decomposition in $T$ if and only if it is a decomposition in $T(A)$.
\end{lemme}
\begin{proof}
Assume $I=(a_i\hat{~}b_i)_{i\in \mathcal I}$ is a decomposition, then it follows immediately from the definition that it is a decomposition from the point of view of $T(A)$.

For the converse, use the internal characterization and strong base change 2 (Lemma \ref{limittype2}) as in the proof of Cororllary \ref{basechange}.
\end{proof}

\begin{lemme}\label{decsequence}
Assume that $K=(a_i\hat{~}b_i)_{i\in \mathcal I}$ is a decomposition of $I=(a_i)_{i\in \mathcal I}$. Let $\mathfrak c$ be a cut of $K$ filled by a sequence $L$ and denote by $\mathfrak c'$ the corresponding cut in $(b_i)_{i\in \mathcal I}$. Let $L_2$ the projection on $L$ on the second factor (so $L_2$ is a totally indiscernible sequence). Let $d\in \monster$ be such that $K$ is indiscernible over $d$ and $L_2$ is a Morley sequence of the limit type of $\mathfrak c'$ over $Kd$. Then $K \cup L$ is indiscernible over $d$ (where $L$ is placed in the cut $\mathfrak c$).
\end{lemme}
\begin{proof}
Assume $L$ is dense of size $|T|$ and using Corollary \ref{fcofcor1}, increase $L$ to some saturated sequence $L'$ filling $c$ and such that the sequence $K_0=K\cup (L'\setminus L)$ is indiscernible over $d$. Let now $a_1\hat{~}b_1 \in L$. This tuple fills a Dedekind cut of $K_0$. By domination in the sequence $K_0$, we see that $K_1=K_0\cup \{a_1\hat{~}b_1\}$ is indiscernible over $d$. Then we can take some other $a_2 \hat{~}b_2 \in L$. It fills a Dedekind cut of $K_1$ and by domination in $K_1$, $K_2 = K_1 \cup \{a_1 \hat{~} b_1\}$ is indiscernible over $d$. Iterating, we see that $K\cup L'$ is indiscernible over $d$ and therefore $K\cup L$ is indiscernible over $d$.
\end{proof}

\begin{lemme}\label{inverselemme}
Let $M$ be a $|T|^+$-saturated model and $p,q\in S(M)$ be two commuting invariant types. Take $I\models p^{(\omega)}$ and any $b\models q$. Then we may find two sequences $I_1,I_2$ such that $I_1+I+I_2$ is a Morley sequence of $p$ over $M$ and $I_1+I_2$ is a Morley sequence of $p$ over $Mb$.
\end{lemme}
\begin{proof}
Let $r$ be the inverse of $p$ over $M$ (recall the definition as stated after Lemma \ref{stationnaire}). We take $I_2$ to be a Morley sequence of $p$ over $MIb$ and then $I_1$ to be a Morley sequence of $r$, indexed in the opposite order, over $MII_2b$. Over $M$, the Morley sequence of $r$ is the opposite of the Morley sequence of $p$ so the first statement follows. To see the second statement, recall that if $s\in S(M)$ is any invariant type, then $r_x \otimes s_y|_M = s_y \otimes p_x|_M$. In particular, $$r_x \otimes (q_y \otimes p^{(n)}_{x_1,...,x_n})|_M = (q_y \otimes p^{(n)}_{x_1,...,x_n}) \otimes p_x |_M = q_y \otimes p^{(n+1)}_{x,x_1,...,x_n}|_M.$$
The result follows.

\end{proof}

\begin{prop}
Assume that all sequences are decomposable, then $T$ is sharp.
\end{prop}
\begin{proof}
Let $M$ be $|T|^+$-saturated and $p\in S(M)$ be an $A$-invariant type. Let $a\models p$. Let $I\subset M$ be a small dense Morley sequence of $p$ over $A$ and let $K\subset M$ be a decomposition of $I$. Let $\mathfrak c$ be a Dedekind cut of $K$ and $\mathfrak c_1$ the corresponding cut of $I$. As in the proof of the moving away lemma \ref{movingaway}, construct some dense sequence $\bar d$ realizing a power of $\lim(\mathfrak c_1^+/M)$ and such that $\bar d$ s-dominates $a$ over $M$. Extend $\bar d$ to $\bar c$ realizing a power of $\lim(\mathfrak c^+/M)$. So $\bar c$ is the union of $\bar d$ and some totally indiscernible sequence $\bar e$. The type of $\bar e$ over $M$ is generically stable.
\\
\underline{Claim :} $\bar e$ s-dominates $a$ over $M$.
\\
Proof : Let $u\in \monster$ be distant from $a$ and independent from $\bar e$ over $M$. Let $u_*$ realize an invariant type distant from $a\bar c$ over $M$ such that $u_*$ s-dominates $u$ and is independent from $\bar c$ over $M$. If we show that $u_* \downfree_M \bar d$, then as $\bar d$ s-dominates $a$ it will follow that $u_* \downfree_M a$ and therefore $u \downfree_M a$. Replacing $u$ by $u_*$, we may now assume that $u$ is distant from $a\bar c$ over $M$ and realizes an invariant type.

Call $r= \lim(\mathfrak c^+)$ (a global invariant type). By Lemma \ref{inverselemme}, let $I_1$ and $I_2$ be two sequences such that $I_1+I_2$ is indiscernible over $Mu$ and $I_1+\bar c+I_2$ is indiscernible over $M$. Also as $u$ is independent from $\bar e$ over $M$, the hypothesis of Lemma \ref{decsequence} are satisfied (where $L_2$ there is $\bar e$ here). We conclude that $u$ is independent from $\bar d$ over $M$ and therefore $u$ is independent from $a$ over $M$.
\end{proof}

\subsection{Reduction to dimension 1}

We prove here the following proposition.

\begin{prop}\label{dim1dec}
Assume that all sequences of elements are decomposable, then every sequence is decomposable.
\end{prop}

Assume from now on that all indiscernible sequences of elements of $\monster$ are decomposable. We will take an arbitrary indiscernible sequence and build a decomposition for it adjoining totally indiscernible sequences to it one-by-one. The proof is an adaptation of the one from Section \ref{constrA}. We start with a base set of parameters $A$ that we allow to grow freely during the construction. In what follows, we work over $A$, even when not explicitly mentioned. We have an indiscernible sequence $I=\llg a_i\hat{~}\alpha_i: i\in \mathcal I\rrg$, where $\mathcal I=(0,1)$ for simplicity and such that the sequence $\llg \alpha_i : i\in \mathcal I\rrg$ is totally indiscernible.

For every $i\in \mathcal I$, call $\mathfrak c_i$ the cut ``$i^+$" of $I$ and $\mathfrak c_i'$ the associated cut in the sequence $\llg \alpha_i : i\in \mathcal I\rrg$.

\vspace{5pt}
\noindent
\un{Step 1} : Derived sequence

\vspace{4pt}
\noindent
Assume we have a witness of non-decomposition in the following form :
\begin{itemize}
\item A tuple $b \in \monster$, some $j \in (0,1)$ and a pair $(a,\alpha)$ such that :
\item $a\hat{~}\alpha$ fills the cut $\mathfrak c_j$ of $I$,
\item $I$ is $b$-indiscernible,
\item $\alpha$ realizes the type $\lim(\mathfrak c_j')$ over $Ib$,
\item $I$ with $a_j\hat{~}\alpha_j$ replaced by $a\hat{~}\alpha$ is not indiscernible over $b$.
\end{itemize}

As in Section \ref{constrA}, adding parameters to the base, we may assume that $b$ is a single point, and that $\tp(a\hat{~}\alpha/b) \neq \tp(a_j\hat{~}\alpha_j/b)$. Let $r=\tp(a\hat{~}\alpha,b)$.

We construct a new sequence $\langle a_i'\hat{~}\alpha_i': i\in \mathcal I\rangle$ such that :
\begin{itemize}
\item $a_i'\hat{~}\alpha_i'$ fills the cut $\mathfrak c_i$ of $I$,
\item $tp(a_i'\hat{~}\alpha_i',b)=r$ for each $i$,
\item The sequence $(\alpha_i')_{i\in \mathcal I}$ realizes $\bigotimes_{i\in \mathcal I} \lim(\mathfrak c_i')$ over $Ib$,
\item The sequence $\llg a_i\hat{~}\alpha_i\hat{~}a_i'\hat{~}\alpha_i' : i\in \mathcal I\rrg$ is $b$-indiscernible.
\end{itemize}
This is possible by indiscernibility of $(a_i\hat{~}\alpha_i)_i$ over $b$ (first pick the points $\alpha_i'$ then choose the $a_i$ filling the cuts and then extract).

\vspace{5pt}
\noindent
\un{Step 2} : Constructing an array

\vspace{4pt}
\noindent
Using Lemma \ref{limittype2}, iterate this construction to obtain an array $\langle a_i^n\hat{~}\alpha_i^n:i\in \mathcal I, n<\omega\rangle$ and sequence $\langle b_n:n<\omega \rangle$ such that :
\begin{itemize}
\item $a_i^0\hat{~}\alpha_i^0=a_i\hat{~}\alpha_i$ for each $i$,
\item For each $i\in \mathcal I$, $0<n<\omega$, the tuple $a_i^n\hat{~}\alpha_i^n$ realizes $\lim(\mathfrak c_i)$ over $\langle b_k,a_i^k\hat{~}\alpha_i^k : i\in \mathcal I, k<n\rangle$,
\item For each $0<n<\omega$, the sequence $(\alpha_i^n)_{i\in \mathcal I}$ realizes the type $\bigotimes_{i\in \mathcal I} \lim(\mathfrak c_i')$ over $\langle b_k,a_i^k\hat{~}\alpha_i^k : i\in \mathcal I, k<n\rangle$,
\item For each $0<n<\omega$, $tp(b_n, \langle a_i^n\hat{~}\alpha_i^n : i\in\mathcal I \rangle/I)=tp(b, \langle a_i'\hat{~}\alpha_i':i\in \mathcal I \rangle /I)$.
\end{itemize}
\noindent
\underline{Claim} : For every $\eta : \mathcal I_0\subset \mathcal I \rightarrow \omega$ injective, the sequence $\langle a_i^{\eta(i)}\hat{~}\alpha_i^{\eta(i)}:i\in \mathcal I_0 \rangle$ is indiscernible, of same EM-type as $I$.

The sequence $U = \langle \alpha_i^n : (i,n)\in \mathcal I\times \omega\rangle$, where $\mathcal I \times \omega$ is ordered lexicographically, is totally indiscernible.
\begin{proof}
Easy, by construction.
\end{proof}

Expanding and extracting, we may assume that the sequence of rows $\langle b_n + ( a_i^n\hat{~}\alpha_i^n )_{i\in \mathcal I}: 0<n<\omega \rangle$ is indiscernible. By assumption all sequences of points are decomposable. So let $(b_n\hat{~}\beta_n)_{n <\omega}$ be an decomposition of $(b_n)_{n<\omega}$. Expanding and extracting again, we may assume that the new sequence of rows $\langle b_n\hat{~}\beta_n + ( a_i^n\hat{~}\alpha_i^n )_{i\in \mathcal I}: 0<n<\omega \rangle$ is indiscernible and that the sequence of columns $\langle( a_i^n\hat{~}\alpha_i^n)_{0<n<\omega}:i\in \mathcal I\rangle$ is indiscernible over $\{b_n\hat{~}\beta_n:n<\omega\}$.

\vspace{5pt}
\noindent
\un{Step 3} : Conclusion

\vspace{4pt}
\noindent

\un{Claim} : The sequences $(b_n\hat{~}\beta_n)_{n<\omega}$ and $\langle (a_i^n\hat{~}\alpha_i^n)_{i \in \mathcal  I} : 0<n<\omega \rangle$ are weakly linked.

The sequences $(b_n\hat{~}\beta_n)_{n<\omega}$ and $U$ are mutually indiscernible.
\begin{proof}
For the first statement, the proof is the same is in Section \ref{constrA}.

The second statement is similar. If for example we have $\phi(b_n,\beta_n,\alpha_i^n)$, then $\phi(b_n,\beta_n,\alpha_j^n)$ must hold for all $j\in \mathcal I$, and therefore by total indiscernibility of $U$ and $NIP$, $\phi(b_n,\beta_n,\alpha_j^m)$ must hold for every $(j,m)\in \mathcal I \times \omega$.
\end{proof}

Let $(c_n,\gamma_n)=(a_{1-\frac 1 n}^n,\alpha_{1-\frac 1 n}^n)$, then :
\begin{enumerate}
\item The sequence $(c_n\hat{~}\gamma_n)_{n<\omega}$ is indiscernible, with same EM-type as $I$;
\item The sequences $(\gamma_n)_{n<\omega}$ and $(b_n\hat{~}\beta_n)_{n<\omega}$ are mutually indiscernible;
\item The sequences $(c_n\hat{~}\gamma_n)_{n<\omega}$ and $(b_n\hat{~}\beta_n)_{n<\omega}$ are weakly linked;
\item We have $\tp(c_n\hat{~}\delta_n,b_m)=r$ if and only if $n=m$.
\end{enumerate}

Consider the indiscernible sequence $(c_n\hat{~}\gamma_n\hat{~}b_n\hat{~}\beta_n)_{n<\omega}$. We may increase it to an indiscernible sequence $(c_n\hat{~}\gamma_n\hat{~}b_n\hat{~}\beta_n)_{n\in \mathcal I}$. Take some $n_0 \in \mathcal I$ and set $\mathcal I = \mathcal I_1+\{n_0\}+\mathcal I_2$. Then by point 3 above, the sequence $\llg b_n\hat{~}\beta_n:n\in \mathcal I_1+\mathcal I_2\rrg$ is indiscernible over $c_{n_0}\hat{~}\gamma_{n_0}$. Therefore point 4 and the definition of decomposition imply that $\beta_{n_0}$ does not realize the limit type of $\llg \beta_n : n\in \mathcal I_1\rrg$ over $\{ b_n\hat{~}\beta_n:n\in \mathcal I_1+\mathcal I_2\} \cup \{c_{n_0}\hat{~}\gamma_{n_0}\}$. Adding parameters to the base, we may assume that it does not realize that limit type over $c_{n_0}\hat{~}\gamma_{n_0}$.

We then iterate the construction, starting with the sequence $(c_n\hat{~}\gamma_n\hat{~}\beta_n)_{n\in \mathcal I}$. Assume that we can do this $|T|^+$ steps. We have at the end some base set of parameters $A$, an $A$-indiscernible sequence $\llg c_n\hat{~}(\alpha^i_n:i<|T|^+):n<\omega\rrg$ (we replaced the index set $\mathcal I$ by $\omega$ for convenience) such that for each $i<|T|^+$, the sequence $(\alpha^i_n)_{n<\omega}$ is totally indiscernible over $A\cup \{ \alpha^j_n : n<\omega, j\neq i\}$ but not indiscernible over $A\cup \{c_n, n<\omega\} \cup \{\alpha^j_n : n<\omega,j< i\}$. By Fodor's lemma, removing some sequences $(\alpha^i_n)_{n<\omega}$ and adding them to $A$, we may assume that for every $i$, $(\alpha^i_n)_{n<\omega}$ is not indiscernible over $A\cup \{d_n, n<\omega\}$. But this contradicts Proposition \ref{weight}.

Therefore this construction must stop after less than $|T|^+$ stages. At the end, we obtain a decomposition of the sequence we started with. This proves Proposition \ref{dim1dec}.

\begin{cor}
Every dp-minimal theory is sharp.
\end{cor}

\begin{ex}[Non-sharp theory]
Let $L_0$ be the language $\{R_n(x,y) :n<\omega\}$ and construct an $L_0$ structure $M_0$ as follows: the universe of $M_0$ is $\mathbb Q$, the ordinary rational numbers, and for every $x,y \in M_0$, $M_0\models R_n(x,y)$ if and only if $x<y \wedge |x-y| < n$ holds in $\mathbb Q$. Let $T_0=Th(M_0)$. Non-realized 1-types over $M_0$ satisfying $R_n(x,a)$ for some $n<\omega$ and $a\in M$ are in natural bijection with cuts of $(\mathbb Q,<)$. In addition to these, there is just one non-realized type $p\in S_1(M_0)$ which satisfies $\neg R_n(x,a)$ for every $n<\omega$ and $a\in M$. This type $p$ is generically stable (and $\emptyset$-invariant). One can check easily that $T_0$ is dp-minimal.

Now consider $L_1 = L_0\cup \{\prec\}$ where $\prec$ is a new binary relation. We expand $M_0$ to an $L_1$-structure $M_1$ by making $\prec$ into a generic order ({\it i.e.}, every $L_0$-infinite definable set of $M_1$ is dense co-dense with respect to $\prec$. See for example \cite{ShSim}). A 1-type over $M_1$ is determined by its reduct to $L_0$ plus its $\prec$-cut. Let $T_1=Th(M_1)$. Easily, $T_1$ eliminates imaginaries so there are no generically stable types (because the structure is linearly ordered). However $T_1$ is not distal: consider $I=(a_i)_{i\in \mathcal I}$ to be a dense $\prec$-increasing sequence of points such that $\neg R_n(a_i,a_j)$ holds for every $n<\omega$ and $i, j\in \mathcal I$. Then this sequence is indiscernible and not distal. To see this, take two cuts $\mathfrak c_1$ and $\mathfrak c_2$ of $I$. Then there is $a$ filling $\mathfrak c_1$ and $b$ filling $\mathfrak c_2$ such that $R_1(a,b)$ holds. The generically stable type $p$ in the reduct is detected by the non-distality of $I$.

We see however, that there is a natural ultra-imaginary stable sort: the quotient of $M$ by the $\bigvee$-definable relation $E=\bigvee_{n<\omega} R_n$. And every point is in some sense s-dominated by its definable closure in that sort. It would be interesting to know if something like this is always true.

 \end{ex}

\bibliographystyle{model1b-num-names}

\bibliography{distal}

\end{document}